\documentclass[reqno,11pt]{amsart}

\usepackage[top=2.0cm,bottom=2.0cm,left=3cm,right=3cm]{geometry}
\usepackage{amsthm,amsmath,amssymb,dsfont}
\usepackage{mathrsfs,amsfonts,functan,extarrows,mathtools}
\usepackage[colorlinks]{hyperref}

\usepackage{marginnote}
\usepackage{xcolor}
\usepackage{stmaryrd}
\usepackage{esint}
\usepackage{graphicx}
\usepackage{bm}

\newtheorem{theorem}{Theorem}[section]
\newtheorem{proposition}{Proposition}[section]
\newtheorem{definition}{Definition}[section]

\newtheorem{remark}{Remark}[section]
\newtheorem{lemma}{Lemma}[section]

\numberwithin{equation}{section}
\allowdisplaybreaks

\def\d{\mathrm{d}}
\def\no{\nonumber}
\def\R{\mathbb{R}}

\def\eps{\varepsilon}

\def\u{\mathrm{u}}

\def\l{\left\langle}
\def\r{\right\rangle}

\def\inn{{\mathrm{in}}}
\def\a{\alpha}

\newcounter{wronumber}

\begin{document}
\title[Non-isothermal PNPF model]
			{Stability of equilibria to the model for non-isothermal electrokinetics}

\author[Ning Jiang]{Ning Jiang}
\address[Ning Jiang]{\newline School of Mathematics and Statistics, Wuhan University, Wuhan, 430072, P. R. China}
\email{njiang@whu.edu.cn}

\author[Yi-Long Luo]{Yi-Long Luo}
\address[Yi-Long Luo]
{\newline Department of Mathematics, The Chinese University of Hong Kong, Hong Kong, 999077, P. R. China}
\email{yl-luo@whu.edu.cn}

\author[Xu Zhang]{Xu Zhang}
\address[Xu Zhang]{\newline School of Mathematics and Statistics, zhengzhou University, Zhengzhou, 450001, P. R. China}
\email{xuzhang889@zzu.edu.cn}

\thanks{ \today}

\maketitle

\begin{abstract}

Recently, energetic variational approach was employed to derive models for non-isothermal electrokinetics by Liu et. al \cite{Liu-Wu-Liu-CMS2018}. In particular, the Poisson-Nernst-Planck-Fourier (PNPF) system for the dynamics of $N$-ionic species in a solvent was derived. In this paper we first reformulate PNPF ($4N+6$ equations) into an evolutional system with $N+1$ equations, and define a new total electrical charge. We then prove the constant states are stable provided that they are such that the perturbed systems around them are dissipative. However, not all positive constant solutions of PNPF are such that the corresponding perturbed systems are dissipative. We characterize a set of equilibria $\mathcal{S}_{eq}$ whose elements satisfy the conditions {(A1)} and {(A2)}, and prove it is nonempty. After then, we prove the stability of these equilibria, thus the global well-posedness of PNPF near them. \\
	
\end{abstract}





\section{Introduction}

\subsection{The Poisson-Nernst-Planck-Fourier system}

The Poisson-Nernst-Planck (PNP) system is one of the most extensively studied models for the transport of charged particles in many physical and biological problems, such as free moving electrons in semiconductors \cite{Jerome-1995, Markowich-1986, M-R-S-1990}, fuel cell \cite{N-P-2007, P-S-SIAM2001}, ion particles in electrokinetic fluids \cite{Ben-Chang-JFM2002, Hunter, J-J-G-B-A-2007, Lyklema}, and ion channels in cell membranes \cite{BTA-2004, Eisenberg-1996, NCE-1999}. The ionic transport can be modeled through PNP theory and its various modified versions \cite{Eisenberg-Liu-SIMA2007, G-N-E-2002, Hsien-HLLL-2015, Liu-E-JCP2014, Qiao-Tu-Lu-JCP2014, Wei-ZCX-SIAM2012, Xu-Ma-Liu-2014}. Through the energetic variational approach. Liu et. al. derived the modified PNP equations with given free energy functional and the form of entropy production \cite{Horng-Lin-Liu-2012, Hyon-E-Liu-CMS2011, Xu-Sheng-Liu-CMS2014}. However, these models are all isothermal: the temperature is fixed as a constant. For this reason, in \cite{Liu-Wu-Liu-CMS2018}, Liu et al. proposed a general framework to derive the transport equations with heat flow through the Energetic Variational Approach. According to the first law of thermodynamics, the total energy is conserved and one can use the Least Action Principle to derive the conservative forces. From the second law of thermodynamics, the entropy increases and the dissipative forces can be computed through the Maximum Dissipation Principle. Combining these two laws, they then conclude with the force balance equations and a temperature equation. In particular, they derived the following PNP equations coupled with the dynamics of temperature equation, which is named Poisson-Nernst-Planck-Fourier (PNPF) system.
\begin{equation}\label{PNPF-1}
  \left\{
    \begin{array}{l}
      \partial_t \rho_i + \nabla \! \cdot \! (\rho_i \u_i) = 0 \,, \\[2mm]
      \nu_i \rho_i ( \u_i - \u_0 ) = - k_B \nabla ( \rho_i T ) - z_i \rho_i \nabla \phi \,, \\[2mm]
      - \eps \Delta \phi = \sum_{j=1}^N z_j \rho_j \,, \ \lim_{|x| \rightarrow + \infty} \phi = 0 \,, \\[2mm]
      \Big( \sum_{i=0}^N k_B c_i \rho_i \Big) \partial_t T + \Big( \sum_{i=0}^N k_B c_i \rho_i \u_i \Big) \cdot \nabla T + \Big( \sum_{i=1}^N k_B \rho_i \nabla \! \cdot \! \u_i \Big) T \\[2mm]
      \qquad = k \Delta T + \lambda_0 |\nabla \u_0 |^2 + \sum_{i=1}^N \nu_i \rho_i |\u_i - \u_0|^2 \,, \\[2mm]
      \lambda_0 \Delta \u_0 = \nabla P_0 + \sum_{i=1}^N \nu_i \rho_i (\u_0 - \u_i) \,, \lim_{|x| \rightarrow + \infty} \u_0 = 0 \,, \\[2mm]
      \qquad \qquad \nabla \! \cdot \! \u_0 = 0 \,,
    \end{array}
  \right.
\end{equation}
for $i = 1, 2, \cdots, N$, which describes the charge dynamics with $N$ $(N \geq 2)$ ionic species. The index $i = 0$ stands for the solvent particles, which is incompressible with constant density $\rho_0 > 0$, and index $1, \cdots , N$ represents the solute species. The time and space variables $(t,x) \in \R^+ \times \R^3$. Since there are many unknown functions and physical constants, for the convenience of readers, we list them in the following tabular form:

\begin{center}
\begin{tabular}{|l|l|}
  \hline
  $\rho_i (t,x)$ & the local density distribution for $i$-th species for $i=1,\cdots,N$ \\
  $\u_i (t,x)$ & the velocity field of the $i$-th species for $i=1,\cdots,N$ \\
  $\u_0 (t,x)$ & the velocity field of the solvent particles \\
  $P_0 (t,x)$ & the Lagrange multiplier corresponding to the incompressibility of the solvent \\
  $\phi (t,x) $ & the mean electrical potential \\
  $T(t,x)$ & the temperature \\
  \hline
\end{tabular}
\end{center}

\begin{tabular}{|l|l|}
  \hline
  $z_i$ & the valences of the $i$-th species for $i=1,2,\cdots, N$ \\
  $k_B$ & the Boltzmann constant \\
  $\nu_i$ & the viscosity between the $i$-th particles and the solvent for $i=1,\cdots , N$ \\
  $k$ & the constant relating with the heat conductance \\
  $\eps$ & the dielectric constant \\
  $\lambda_0$ & the shear viscosity coefficient for the solvent \\
  $\rho_0$ & the constant density of the solvent \\
  $c_0$ & the constant related to the heat capacitance of the solvent \\
  $c_i$ & the constant related to the heat capacitance of the $i$-th species for $i=1, \cdots, N$ \\
  \hline
\end{tabular}\\

We further give some assumptions on the all coefficients throughout this paper. To cover the most general case,  the valences $z_i$ can be assumed
\begin{equation}\label{Asmp-z}
  \begin{aligned}
    z_1 \leq z_2 \leq \cdots \leq z_l < 0 < z_{l+1} \leq z_{l+2} \leq \cdots \leq z_N
  \end{aligned}
\end{equation}
for some integer $1 \leq l \leq N - 1$. Moreover, the other coefficients are all naturally considered to be positive, say,
\begin{equation}\label{Asmp-othCoefs}
  \begin{aligned}
    k_B \,, \ \nu_i \,, \ k \,, \ \eps \,, \ \lambda_0 \,, \ \rho_0 \,, \ c_0 \,, \ c_i > 0 \,,
  \end{aligned}
\end{equation}
where $i = 1, 2, \cdots, N$.

We emphasize that (as pointed out in \cite{Liu-Wu-Liu-CMS2018}) we cannot simply assume $\u_0$ is a constant, since the solvent energy and entropy are included. This is different from the original PNP equations where the
velocity, energy and entropy of the solvent are not considered. One also observes from \eqref{u0} in Subsection \ref{Subsec-Reform} below that $\u_0$ is determined by $\rho_i$ and generally not constant. The system \eqref{PNPF-1} might not be solvable without the solvent viscosity $\lambda_0$. Simply letting the temperature $T$ to be constant will {\em not} cover the original PNP system. When the temperature $T$ is constrained to be constant, the temperature equation in \eqref{PNPF-1} reduces to a nontrivial equation
\begin{equation*}
  \begin{aligned}
    \sum_{i=1}^N T k_B \rho_i \nabla \cdot \u_i = \lambda_0 |\nabla \u_0|^2 + \sum_{i=1}^N \nu_i \rho_i |\u_i - \u_0|^2 \,,
  \end{aligned}
\end{equation*}
which will be such that the system is overdetermined. This can be observed intuitively through the reformulations \eqref{Evol-rho-i}-\eqref{Evol-T} in Subsection \ref{Subsec-Reform} below, which consists of $N+1$ equations associated with $N+1$ unknowns (the necessary features of a determined system). Even the temperature $T$ is assumed to be constant, the equation \eqref{Evol-T} does not vanish, which will be such that the reduced system from \eqref{Evol-rho-i}-\eqref{Evol-T} is overdetermined. We consequently fail to cover the original PNP system by simply letting the temperature $T$ to be constant. In this sense, the system PNPF is not just simply adding a temperature equation comparing to PNP system.

All the constant states are trivial solutions of PNPF \eqref{PNPF-1}. So a natural question is about the long time stability of these equilibria, in other words, the global in time well-posedness around them. We discover that not all constant are stable, in the sense that the PNPF linearized around the constant states are not necessarily dissipative. The main novelty of this paper is that we characterize the set of the equilibria around which the PNPF are stable. We prove this set is nonempty under some coefficients assumptions, and analyze their long time stability, and thus prove the global in time well-posedness around them.

\subsection{Reformulation of PNPF system}\label{Subsec-Reform}

The system \eqref{PNPF-1} looks complicated (with $4N+6$ unknown functions). However, it only contains $N+1$ evolutional equations. It is just like 3-$D$ incompressible Navier-Stokes equations, which include 4 equations, but there is no evolutional equation for the pressure. A traditional treatment is using Leray projection to rewrite them into 3 equations only for velocity and solve the pressure by a Poisson equation. We can treat the PNPF here in the same spirit, although the process will be more tedious. In fact, PNPF can be transformed to an evolutional system with only $N+1$ equations for $\rho_i (t,x)$ $(i=1,\cdots,N)$ and $T (t,x)$. First, from the Poisson equation in \eqref{PNPF-1}, we know
\begin{equation}\label{phi}
  \begin{aligned}
    \phi = \frac{1}{\eps} (- \Delta)^{-1} \Big( \sum_{j=1}^N z_j \rho_j \Big) \,,,
  \end{aligned}
\end{equation}
which gives
\begin{equation}\label{ui-u0}
  \begin{aligned}
    \u_i - \u_0 = - \frac{k_B}{\eps \nu_i \rho_i} \nabla (\rho_i T) - \frac{z_i}{\nu_i} \nabla (- \Delta)^{-1} \Big( \sum_{j=1}^N z_j \rho_j \Big) \,.
  \end{aligned}
\end{equation}
Let $\mathcal{P}$ be the usual Leray projection, then the second and the last two equations of \eqref{PNPF-1} imply that
\begin{equation}\label{u0}
  \begin{aligned}
    \u_0 = - (- \Delta )^{-1} \left\{ \sum_{i=1}^N \mathcal{P} \left[ \frac{z_i}{\eps \lambda_0} \rho_i \nabla ( - \Delta )^{-1} \left( \sum_{j=1}^N z_j \rho_j \right) \right] \right\} \,,
  \end{aligned}
\end{equation}
and
\begin{equation}\label{P0}
  \begin{aligned}
    P_0 = - \sum_{i=1}^N k_B \rho_i T + (- \Delta )^{-1} \nabla \! \cdot \! \left[ \sum_{i=1}^N \frac{z_i}{\eps} \rho_i \nabla (- \Delta )^{-1} \left( \sum_{j=1}^N z_j \rho_j \right) \right] \,.
  \end{aligned}
\end{equation}
Thus, we have
\begin{equation}\label{ui}
  \begin{aligned}
    \u_i = & - \frac{k_B}{\eps \nu_i \rho_i} \nabla (\rho_i T) - \frac{z_i}{\nu_i} \nabla (- \Delta)^{-1} \Big( \sum_{j=1}^N z_j \rho_j \Big) \\
    & - (- \Delta )^{-1} \left\{ \sum_{i=1}^N \mathcal{P} \left[ \frac{z_i}{\eps \lambda_0} \rho_i \nabla ( - \Delta )^{-1} \left( \sum_{j=1}^N z_j \rho_j \right) \right] \right\} \,.
  \end{aligned}
\end{equation}
Then, the first $N$ evolutions of $\rho_i$ in \eqref{PNPF-1} read
\begin{equation}\label{Evol-rho-i}
  \begin{aligned}
    \partial_t \rho_i - \frac{k_B}{\eps \nu_i} \Delta ( \rho_i T ) = & \nabla \! \cdot \! \left[ \frac{z_i \rho_i}{\nu_i} \nabla (- \Delta)^{-1} \Big( \sum_{j=1}^N z_j \rho_j \Big)  \right] \\
    & + \nabla \! \cdot \! \left\{ \rho_i (- \Delta)^{-1} \left[ \sum_{i=1}^N \mathcal{P} \left( \frac{z_i \rho_i}{\eps \lambda_0} \nabla (- \Delta )^{-1} \Big( \sum_{j=1}^N z_j \rho_j \Big) \right) \right] \right\}
  \end{aligned}
\end{equation}
for $i = 1, \cdots , N$. Moreover, we can deduce from plugging the relations \eqref{ui-u0}, \eqref{u0} and \eqref{ui} into the forth equation of \eqref{PNPF-1} that
  \begin{align}\label{Evol-T}
    \no \Big( \sum_{i=0}^N k_B c_i \rho_i \Big) \partial_t T - k \Delta T = \sum_{i=1}^N \nu_i \rho_i \left| \frac{k_B}{\eps \nu_i \rho_i} \nabla (\rho_i T) + \frac{e z_i}{\nu_i} \nabla (- \Delta)^{-1} \Big( \sum_{j=1}^N z_j \rho_j \Big) \right|^2 \\
    \no + \lambda_0 \left| \nabla (- \Delta )^{-1} \left\{ \sum_{i=1}^N \mathcal{P} \left[ \frac{z_i}{\eps \lambda_0} \rho_i \nabla ( - \Delta )^{-1} \left( \sum_{j=1}^N z_j \rho_j \right) \right] \right\} \right|^2 \\
    \no + \sum_{i=0}^N \left( \frac{k_B^2 c_i}{\eps \nu_i} \nabla (\rho_i T) + \frac{k_B c_i z_i}{\nu_i} \rho_i \nabla (- \Delta)^{-1} \Big( \sum_{j=1}^N z_j \rho_j \Big) \right) \cdot \nabla T \\
    \no + \sum_{i=0}^N k_B c_i \rho_i (- \Delta )^{-1} \left\{ \sum_{i=1}^N \mathcal{P} \left[ \frac{z_i}{\eps \lambda_0} \rho_i \nabla ( - \Delta )^{-1} \left( \sum_{j=1}^N z_j \rho_j \right) \right] \right\} \cdot \nabla T \\
    \no + \sum_{i=1}^N k_B \rho_i \nabla \! \cdot \! \left[ \frac{k_B}{\eps \nu_i \rho_i} \nabla (\rho_i T) + \frac{z_i}{\nu_i} \nabla (- \Delta)^{-1} \Big( \sum_{j=1}^N z_j \rho_j \Big) \right] T \\
    + \sum_{i=1}^N k_B \rho_i \nabla \! \cdot \! \left\{ (- \Delta )^{-1} \left\{ \sum_{i=1}^N \mathcal{P} \left[ \frac{z_i}{\eps \lambda_0} \rho_i \nabla ( - \Delta )^{-1} \left( \sum_{j=1}^N z_j \rho_j \right) \right] \right\} \right\} T \,.
  \end{align}
One notices that the equation \eqref{Evol-rho-i} coupled the evolution \eqref{Evol-T} is a closed system associated with the unknown functions $\rho_i$ and $T$, which can be solved under the following initial conditions
\begin{equation}\label{IC-PNPF}
  \begin{aligned}
    \rho_i (0, x) = \rho_i^\inn (x) \,, \ i = 1,2, \cdots, N \,, \ T (0, x) = T^\inn (x) \,.
  \end{aligned}
\end{equation}
As for the initial  data for $\phi$, it can be determined by
\begin{equation}\label{IC-PNPF-phi}
  \begin{aligned}
    - \Delta \phi^\inn = \tfrac{1}{\eps} \sum_{i=1}^N z_i \rho_i^\inn \,, \ \lim_{|x| \rightarrow + \infty} \phi^\inn = 0 \,.
  \end{aligned}
\end{equation}

However, the system \eqref{Evol-rho-i}-\eqref{Evol-T} still looks tedious. We can further simplify it. More precisely, Let
$$m = \sum_{j=1}^N z_j \rho_j$$
which is called the {\em total electrical charge}. Then the system \eqref{PNPF-1} can be rewritten as
\begin{equation}\label{PNPF-2}
\left\{
\begin{array}{l}
\partial_t \rho_i + \u_0 \cdot \nabla \rho_i - \tfrac{k_B}{\nu_i} \Delta (\rho_i T) - \nabla \cdot \big( \tfrac{z_i }{\nu_i} \rho_i \nabla \phi \big) = 0 \,, \ i = 1, \cdots , N \,, \\[2mm]
\qquad \qquad - \Delta \phi = \tfrac{1}{\eps} m \,, \\[2mm]
\lambda_0 \Delta \u_0 = \nabla P_0 + \sum_{i=1}^N k_B  \nabla (\rho_i T) + m \nabla \phi \,, \\[2mm]
\qquad \qquad \qquad \nabla \cdot \u_0 = 0 \,, \\[2mm]
\Big( \sum_{i=0}^N k_B c_i \rho_i \Big) \partial_t T - k \Delta T + \sum_{i=0}^N k_B c_i \rho_i \u_0 \cdot \nabla T - \sum_{i=1}^N \tfrac{k_B^2 c_i}{\nu_i} \nabla (\rho_i T) \cdot \nabla T \\[2mm]
= \lambda_0 |\nabla \u_0|^2 + \sum_{i=1}^N \tfrac{1}{\nu_i \rho_i} \left| k_B \nabla (\rho_i T) + z_i \rho_i \nabla \phi \right|^2 + \sum_{i=1}^N \tfrac{k_B c_i z_i}{\nu_i} \rho_i \nabla \phi \cdot \nabla T \\[2mm]
\qquad + \sum_{i=1}^N \Big( \tfrac{k_B^2}{\nu_i} \Delta (\rho_i T) - \tfrac{k_B^2}{\nu_i \rho_i} \nabla \rho_i \cdot \nabla (\rho_i T) - \tfrac{k_B z_i}{\eps \nu_i} \rho_i m \Big) T \\[2mm]
\partial_t m + \u_0 \cdot \nabla m - \tfrac{k_B}{\nu} \Delta (m T) = \sum_{i=1}^N \tfrac{z_i^2}{\nu_i} \nabla \cdot ( \rho_i \nabla \phi ) + k_B \sum_{i=1}^N ( \tfrac{1}{\nu_i} - \tfrac{1}{\nu} ) z_i \Delta (\rho_i T ) \,.
\end{array}
\right.
\end{equation}
where $\nu = \tfrac{N}{\sum_{j=1}^N \tfrac{1}{\nu_j}} > 0$ is the harmonic average of the viscosities $\nu_1$, $\nu_2$, $\cdots$, $\nu_N$. Furthermore, if we consider the following perturbations
\begin{equation}\label{Perturbations}
\begin{aligned}
\rho_i = \delta_i + n_i \,, \, i = 1, \cdots , N \,, \, T = 1 + \theta \,,
\end{aligned}
\end{equation}
where $\delta_i > 0$ are arbitrarily fixed constants with the constraint $\sum_{j=1}^N z_j \delta_j = 0$, then the functions $( n_1, \cdots, n_N, \theta, m, \phi,  \u_0, P_0)$ subjects to the following equations
\begin{equation}\label{PNPF-Perturbed}
\left\{
\begin{array}{l}
\partial_t n_i - \tfrac{k_B}{\nu_i} \Delta n_i = \tfrac{k_B \delta_i}{\nu_i} \Delta \theta + \tfrac{z_i \delta_i}{\nu_i} \Delta \phi + R_{n_i} \,, \ i = 1, 2, \cdots, N \,, \\[2mm]
\qquad \qquad \qquad - \Delta \phi = \tfrac{1}{\eps} m \,, \\[2mm]
\lambda_0 \Delta \u_0 = \nabla P_0 + \sum_{i=1}^N k_B \nabla (n_i + \delta_i \theta) + R_{\u_0} \,, \\[2mm]
\qquad \qquad \qquad \nabla \cdot \u_0 = 0 \,, \\[2mm]
a \partial_t \theta - b \Delta \theta = \sum_{i=1}^N \tfrac{k_B^2}{\nu_i} \Delta n_i + \sum_{i=1}^N \tfrac{k_B z_i \delta_i}{ \nu_i} \Delta \phi + R_\theta \,, \\[2mm]
\partial_t m - \tfrac{k_B}{\nu} \Delta m + \tfrac{1}{\eps} \big( \sum_{i=1}^N \tfrac{z_i^2 \delta_i}{\nu_i} \big) m = k_B \sum_{i=1}^N ( \tfrac{1}{\nu_i} - \tfrac{1}{\nu} ) z_i \Delta n_i + k_B \sum_{i=1}^N \tfrac{z_i \delta_i}{\nu_i} \Delta \theta + R_m \,,
\end{array}
\right.
\end{equation}
where
\begin{equation}\label{Coeffs-ab}
\begin{aligned}
a = k_B c_0 \rho_0 + \sum_{i=1}^N k_B c_i \delta_i > 0 \,, \ b = k + \sum_{i=1}^N \frac{k_B^2 \delta_i}{\nu_i} > 0 \,,
\end{aligned}
\end{equation}
and the nonlinear terms $R_{n_i} : = R_{n_i} (n_i, m, \phi, \u_0)$, $R_{\u_0} : = R_{\u_0} ( n_1, \cdots, n_N, \theta, m, \phi) $, $R_\theta : = R_\theta ( n_1, \cdots, n_N, \theta, m, \phi, \u_0) $ and $R_m : = R_m ( n_1, \cdots, n_N, \theta, m, \phi, \u_0)$ are defined as follows:
\begin{align}
\label{R-ni} & R_{n_i} = - \u_0 \cdot \nabla n_i - \tfrac{z_i}{\eps \nu_i} n_i m + \tfrac{z_i}{\nu_i} \nabla n_i \cdot \nabla \phi \,, \\
\label{R-u0} & R_{\u_0} = \sum_{i=1}^N k_B \nabla (n_i \theta ) + m \nabla \phi \,, \\
\label{R-m} & R_m = - \u_0 \cdot \nabla m + k_B \sum_{i=1}^N \tfrac{z_i}{\nu_i} \Delta (n_i \theta) - \tfrac{1}{\eps} \sum_{i=1}^N \tfrac{z_i^2}{\nu_i} n_i m + \sum_{i=1}^N \tfrac{z_i^2}{\nu_i} \nabla n_i \cdot \nabla \phi \,, \\
\label{R-theta} & R_\theta = - a \u_0 \cdot \nabla \theta + \tfrac{a R_\theta^\star }{a + \sum_{i=1}^N k_B c_i n_i} - \tfrac{\sum_{i=1}^N k_B c_i n_i}{a + \sum_{i=1}^N k_B c_i n_i} \Big( b \Delta \theta + \sum_{i=1}^N \tfrac{k_B^2}{\nu_i} \Delta n_i - \sum_{i=1}^N \tfrac{k_B z_i \delta_i}{\nu_i} m \Big)   \,.
\end{align}
Here the term $R_\theta^\star := R_\theta^\star ( n_1, \cdots, n_N, \theta, m, \phi, \u_0)$ is of the form
\begin{equation}\label{R-theta-*}
\begin{aligned}
R_\theta^\star = & \sum_{i=1}^N \tfrac{k_B^2 c_i}{\nu_i} \nabla (n_i + \delta_i \theta + n_i \theta) \cdot \nabla \theta + \sum_{i=1}^N \tfrac{k_B c_i z_i}{\nu_i} (\delta_i + n_i) \nabla \phi \cdot \nabla \theta \\
& + \sum_{i=1}^N \tfrac{k_B^2}{\nu_i} \Delta (n_i \theta) - \sum_{i=1}^N \tfrac{k_B^2 z_i}{\eps \nu_i} n_i m - \sum_{i=1}^N \frac{k_B^2}{\nu_i (\delta_i + n_i)} \nabla n_i \cdot \nabla (n_i + \delta_i \theta + n_i \theta) \\
& + \sum_{i=1}^N \tfrac{k_B^2}{\nu_i} \Delta ( n_i + \delta_i \theta + n_i \theta ) \theta - \sum_{i=1}^N \tfrac{k_B z_i}{\eps \nu_i} (\delta_i + n_i) m \theta \\
& + \lambda_0 |\nabla \u_0|^2 - \sum_{i=1}^N \frac{k_B^2}{\nu_i (\delta_i + n_i)} \theta \nabla n_i \cdot \nabla (n_i + \delta_i \theta + n_i \theta) \\
& + \sum_{i=1}^N \frac{1}{\nu_i (\delta_i + n_i)} \left| k_B \nabla (n_i + \delta_i \theta + n_i \theta) + z_i (\delta_i + n_i) \nabla \phi \right|^2 \,.
\end{aligned}
\end{equation}
The details of the derivations on the forms \eqref{PNPF-2} and \eqref{PNPF-Perturbed} can be referred to Lemma \ref{Lmm-Syst-Transform}.

\begin{remark}
	The evolution of the total electrical charge $m$ in \eqref{PNPF-Perturbed} is not an independent equation, because of the relation $m =\sum_{j=1}^N z_j \rho_j = \sum_{j=1}^N z_j n_j$. However, it has the dissipative effect $\tfrac{k_B}{\nu} \Delta m$ and the damping effect $\tfrac{1}{\eps} \big( \sum_{i=1}^N \tfrac{z_i^2 \delta_i}{\nu_i} \big) m$, which will play an essential role in proving the global in time solutions near the admissible equilibria.
\end{remark}

\subsection{Notations and main results}

To state our results, we collect here some notations. The symbol $A \lesssim B$ represents $A \leq C B$ for some harmless constant $C > 0$. We further denote by $A \thicksim B$ if there are two constants $C_1 , C_2 > 0$, independent of $\eps > 0$, such that $C_1 A \leq B \leq C_2 A$. For convenience, we also denote by
$$ L^p = L^p (\R^3) $$
for all $p \in [ 1, \infty ]$, which endows with the norm $ \| f \|_{L^p} = \left( \int_{\R^3} |f(x)|^p \d x \right)^\frac{1}{p} $ for $p \in [ 1 , \infty )$ and $\| f \|_{L^\infty} = \underset{x \in \R^3}{\textrm{ess sup}} \, |f(x)|$. For $p = 2$, we use the notation $ \langle \cdot \,, \cdot \rangle $ to represent the inner product on the Hilbert space $L^2$.

For any multi-index $ \a = ( \a_1, \a_2, \a_3 )$ in $\mathbb{N}^3$, we denote the $\a$-th partial derivative by
\begin{equation*}
  \partial^\a = \partial^{\a_1}_{x_1} \partial^{\a_2}_{x_2} \partial^{\a_3}_{x_3} \,.
\end{equation*}
If each component of $\a \in \mathbb{N}^3$ is not greater than that of $\tilde{\a}$'s, we denote by $\a \leq \tilde{\a}$. The symbol $\a < \tilde{\a}$ means $\a \leq \tilde{\a}$ and $|\a| < | \tilde{\a} |$, where $|\a|= \a_1 + \a_2 + \a_3$. We define the Sobolev space $H^s = H^s (\R^3)$ by the norm
\begin{align*}
  \| f \|_{H^s} = \bigg(  \sum_{|\a| \leq s} \| \partial^\a f \|^2_{L^2} \bigg)^\frac{1}{2} < \infty \,.
\end{align*}

Now we state our main theorem as follows:

\begin{theorem}\label{Main-Thm}
	Let $s \geq 3$, $N \geq 2$ be any fixed integers, and the coefficients satisfy \eqref{Asmp-z}, \eqref{Asmp-othCoefs} and the further assumption.
	\begin{equation}\label{Key-Asmp}
	  \begin{aligned}
	    \max_{1 \leq i \leq N} ( 1 - \tfrac{\nu_i}{\nu} )^2 < \tfrac{1}{2} \,,
	  \end{aligned}
	\end{equation}
	where $\nu = \tfrac{N}{\sum_{i=1}^N \tfrac{1}{\nu_i}} > 0$ is the harmonic average of the viscosities $\nu_1, \nu_2, \cdots, \nu_N > 0$. Let $(\delta_1, \cdots, \delta_N)$ belong to the equilibria set $\mathcal{S}_{eq}$ given in Definition \ref{Def-AEF}. There is a small constant $\xi_0 > 0$, depending only on $s$, $N$, $\delta_1, \cdots, \delta_N$ and the all coefficients, such that if
	\begin{equation}
	  \begin{aligned}
	    E^\inn : = \sum_{i=1}^N \| \rho_i^\inn - \delta_i \|^2_{H^s} + \| T^\inn - 1 \|^2_{H^s} + \| \nabla \phi^\inn \|^2_{H^s} \leq \xi_0 \,,
	  \end{aligned}
	\end{equation}
	then the Cauchy problem \eqref{PNPF-1}-\eqref{IC-PNPF} admits a unique global solution $(\rho_1, \cdots, \rho_N, T)$,
	\begin{equation*}
	  \begin{aligned}
	    \rho_1 - \delta_1, \cdots, \rho_N - \delta_N, T - 1 \in L^\infty (\R^+; H^s) \,, \ \nabla \rho_1 \, \cdots, \nabla \rho_N, \nabla T \in L^2 (\R^+; H^s) \,.
	  \end{aligned}
	\end{equation*}
	Moreover, there holds
	\begin{equation}\label{Global-Energy}
	  \begin{aligned}
	    & \sup_{t \geq 0} \big( \| \rho_1 - \delta_1 \|^2_{H^s} + \cdots + \| \rho_N - \delta_N \|^2_{H^s} + \| T - 1 \|^2_{H^s} + \| \nabla \phi \|^2_{H^s} \big) \\
	    & + \int_0^\infty \big( \| \nabla \rho_1 \|^2_{H^s} + \cdots + \| \nabla \rho_N \|^2_{H^s} + \| \nabla T \|^2_{H^s} + \| \nabla \u_0 \|^2_{H^s} \big) \d t \leq C_0 E^\inn
	  \end{aligned}
	\end{equation}
	for some constant $C_0 > 0$, depending only on $s$, $N$, $\delta_1, \cdots, \delta_N$ and the all coefficients. Furthermore, the functions $(\phi, \u_0, \u_1, \cdots, \u_N, P_0)$, determined by $(\rho_1, \cdots, \rho_N, T)$ through \eqref{phi}, \eqref{u0}, \eqref{ui} and \eqref{P0}, respectively, satisfy
	\begin{equation*}
	  \begin{aligned}
	    & \nabla \phi, \Delta \phi \in L^\infty (\R^+; H^s) \,, \ \nabla \Delta \phi \in L^2 (\R^+; H^s) \,, \\
	    & \nabla \u_0 \in L^2 (\R^+; H^s) \,, \Delta \u_0 \in L^\infty (\R^+; H^s) \cap L^2_{loc} (\R^+; H^{s+1}) \,, \\
	    & \nabla P_0 \in L^\infty (\R^+; H^{s-1}) \cap L^2_{loc} (\R^+; H^s) \,, \\
	    & \nabla \u_1 , \nabla \u_2, \cdots , \nabla \u_N \in L^2 (\R^+; H^{s-1}) \,.
	  \end{aligned}
	\end{equation*}
\end{theorem}

\begin{remark}
	The first condition $\sum_{i=1}^N z_i \delta_i = 0$ in the equilibria set $\mathcal{S}_{eq}$ means that the stabilities verified in Theorem \ref{Main-Thm} is around the constant equilibrium state with zeroed total electrical charge.
\end{remark}

\begin{remark}
	The assumption \eqref{Key-Asmp} means that the all viscosities $\nu_i$ between the $i$-th particles and the solvent are about the same size, which are conceivably reasonable from the view of physics. Take Sodium chloride solution as an example, the viscosity $\nu_1$ between sodium ion $Na^+$ and the solvent water is $1.334 nm^2/ns$, and the viscosity $\nu_2$ between chloride ion $Cl^-$ and the solvent water is $2.032 nm^2/ns$ (see \cite{Lide-2004}). Then the harmonic average $\nu$ of $\nu_1$ and $\nu_2$ is
	\begin{equation*}
	  \begin{aligned}
	    \nu = \tfrac{2}{\tfrac{1}{\nu_1} + \tfrac{1}{\nu_2}} = \tfrac{1.334 \times 2.032}{1.683} nm^2/ns \,.
	  \end{aligned}
	\end{equation*}
	There therefore hold
	\begin{equation*}
	  \begin{aligned}
	    (1 - \tfrac{\nu_1}{\nu})^2 = (1 - \tfrac{\nu_2}{\nu})^2 = 0.043 < \tfrac{1}{2} \,,
	  \end{aligned}
	\end{equation*}
	which means that the assumption \eqref{Key-Asmp} is satisfied.
\end{remark}

\subsection{Key ideas and sketch of the proofs}

The key observation of this paper is that although all positive constants states are solutions of the system \eqref{PNPF-1}, for the fixed coefficients with the assumptions \eqref{Asmp-z} and \eqref{Asmp-othCoefs}, not all the constant states around which the system of the fluctuations are {\em dissipative}. Here we emphasize that one can prove the constant states are stable provided that they are such that the perturbed systems around them are dissipative. So, we need to find some suitable equilibrium states $(\delta_1, \cdots, \delta_N, 1)$ associated with $(\rho_1, \cdots, \rho_N, T)$, so that the basic energy of the whole system near the equilibrium state is dissipative. Here $\delta_1, \cdots, \delta_N > 0$ are to be determined.

First, from the physical point of view, the total electrical charge $m = \sum_{i=1}^N z_i \rho_i$ is a very important physical quantity in the PNPF system, whose evolution is governed by the last equation of \eqref{PNPF-Perturbed}, namely,
\begin{equation*}
  \begin{aligned}
    \partial_t m - \tfrac{k_B}{\nu} \Delta m + \tfrac{1}{\eps} \big( \sum_{i=1}^N \tfrac{z_i^2 \delta_i}{\nu_i} \big) m = k_B \sum_{i=1}^N ( \tfrac{1}{\nu_i} - \tfrac{1}{\nu} ) z_i \Delta n_i + k_B \sum_{i=1}^N \tfrac{z_i \delta_i}{\nu_i} \Delta \theta + R_m \,,
  \end{aligned}
\end{equation*}
which has the dissipative effect $ - \tfrac{k_B}{\nu} \Delta m $ and the damping effect $\tfrac{1}{\eps} \big( \sum_{i=1}^N \tfrac{z_i^2 \delta_i}{\nu_i} \big) m$. These two structures play an essential role in deriving the global energy bounds. Moreover, together with the Poisson equation $- \Delta \phi = \tfrac{1}{\eps} m$, the above $m$-equation will give us a energy structure of $\phi$
\begin{equation*}
  \begin{aligned}
    \tfrac{1}{2} \tfrac{\d}{\d t} \big( \eps \| \nabla \phi \|^2_{L^2} \big) + \tfrac{k_B \eps}{\nu } \| \Delta \phi \|^2_{L^2} + ( \sum_{i=1}^N \tfrac{z_i^2 \delta_i}{\nu_i} ) \| \nabla \phi \|^2_{L^2}
  \end{aligned}
\end{equation*}
by dot with $\phi$ and integrating by parts over $x \in \R^3$.

Second, in order to see the intrinsic structure of the PNPF system \eqref{PNPF-1}, we linearize the equations \eqref{PNPF-1} near the constant equilibrium state $(\delta_1, \cdots , \delta_N, 1)$, which reduces to the linearized system \eqref{PNPF-Linearized}. We rewrite this linear system as an abstract form
\begin{equation}\label{Linear-Abstract}
  \begin{aligned}
    \partial_t U - D_\delta (U) = L_\delta (U) \,,
  \end{aligned}
\end{equation}
where
  \begin{align*}
    & U = \left(
          \begin{array}{c}
            n_1 \\
            \cdot \\[-3mm]
            \cdot \\[-3mm]
            \cdot \\
            n_N \\
            a \theta \\
            m \\
            \eps \nabla \phi
          \end{array}
        \right) \,, \
    D_\delta (U) = \left(
             \begin{array}{c}
               \tfrac{k_B \delta_1}{\nu_1} \Delta n_1 \\
               \cdot \\[-3mm]
               \cdot \\[-3mm]
               \cdot \\
               \tfrac{k_B \delta_N}{\nu_N} \Delta n_N \\
               b \Delta \theta \\
               \begin{aligned}
                 \tfrac{k_B}{\nu} \Delta m - \tfrac{1}{\eps} \big( \sum_{i=1}^N \tfrac{z_i^2 \delta_i}{\nu_i} \big) m
               \end{aligned} \\
               \begin{aligned}
                 \tfrac{k_B \eps}{\nu } \nabla \Delta \phi - ( \sum_{i=1}^N \tfrac{z_i^2 \delta_i}{\nu_i} ) \nabla \phi
               \end{aligned}
             \end{array}
           \right) \,, \\[2mm]
    & L_\delta (U) = \left(
             \begin{array}{c}
               \tfrac{k_B \delta_1}{\nu_1} \Delta \theta + \tfrac{ z_1 \delta_1}{\nu_1} \Delta \phi \\
               \cdot \\[-3mm]
               \cdot \\[-3mm]
               \cdot \\
               \tfrac{k_B \delta_N}{\nu_N} \Delta \theta + \tfrac{ z_N \delta_N}{\nu_N} \Delta \phi \\
               \begin{aligned}
                 \sum_{i=1}^N \tfrac{k_B^2}{\nu_i} \Delta n_i + \sum_{i=1}^N \tfrac{ k_B z_i \delta_i}{\nu_i} \Delta \phi
               \end{aligned} \\
               \begin{aligned}
                 k_B \sum_{i=1}^N ( \tfrac{1}{\nu_i} - \tfrac{1}{\nu} ) z_i \Delta n_i + k_B \sum_{i=1}^N \tfrac{z_i \delta_i}{\nu_i} \Delta \theta
               \end{aligned} \\
               \begin{aligned}
                 - k_B \sum_{i=1}^N ( \tfrac{1}{\nu_i} - \tfrac{1}{\nu} ) z_i \nabla n_i - k_B \sum_{i=1}^N \tfrac{z_i \delta_i}{\nu_i} \nabla \theta
               \end{aligned}
             \end{array}
           \right) \,.
  \end{align*}
Here the function $\phi$ is determined by $- \Delta \phi = \tfrac{1}{\eps} m$. Although the linear system \eqref{Linear-Abstract} has dissipation mechanism $D_\delta (U)$, whose coefficients depend on the equilibrium state $( \delta_1, \cdots, \delta_N )$, the linear term $L_\delta (U)$ may have a {\em negative impact} on the dissipation mechanism of the entire system near the general equilibrium state. We thereby introduce an equilibria set $\mathcal{S}_{eq}$, which contains all possible equilibrium states such that the linear system \eqref{Linear-Abstract} is dissipative and thus the nonlinear system \eqref{PNPF-1} is also dissipative. Moreover, we can prove the set $\mathcal{S}_{eq}$ is nonempty under the assumption \eqref{Key-Asmp}, i.e., $\max_{1 \leq i \leq N} (1 - \tfrac{\nu_i}{\nu})^2 < \tfrac{1}{2}$, in Proposition \ref{Prop-AEF}. This is the main novelty of current paper.

At the end, based on the linearized dissipative law in Section \ref{Sec:Linear-Dissp}, we employ the energy method to derive the a priori energy estimates given in Proposition \ref{Prop-Apriori}. We emphasize that due to the absence of $L^2$-norm of $\phi$ and $\u_0$, which subject to the nonlinear elliptic equations, in the energy and dissipative rate, there are several key cancellations on $\phi$ and $\u_0$ when controlling the nonlinear terms in $\phi$ and $\u_0$ equations. More precisely, the are \eqref{Key-Cnc-u0} and \eqref{Key-Rlts} below, hence,
  \begin{align*}
    & - \tfrac{1}{\lambda_0} \l m \nabla \phi , \u_0 \r = \tfrac{\eps}{\lambda_0} \l \nabla \phi \otimes \nabla \phi , \nabla \u_0 \r \,, \\
    & \l \u_0 \cdot \nabla \Delta \phi , \phi \r = - \l \nabla \phi \otimes \nabla \phi , \nabla \u_0 \r \,, \\
    & \l n_i \Delta \phi + \nabla n_i \cdot \nabla \phi , \phi \r = - \l n_i , |\nabla \phi|^2 \r \,.
  \end{align*}
Then, by the continuity arguments, we construct the unique global smooth solution near the {dissipative} equilibrium states.

\subsection{Organization of this paper}

In the next section, we study the dissipative structures of the linearized equations \eqref{PNPF-Linearized} of \eqref{PNPF-1}. In order to ensure the dissipation of the whole system, we define the equilibria set $\mathcal{S}_{eq}$, which is nonempty under a further coefficients assumption proved in Proposition \ref{Prop-AEF}. In Section \ref{Sec:Global}, we derive the global a priori estimates and prove the global well-posedness near the equilibrium states in $\mathcal{S}_{eq}$ by employing the continuity arguments. Finally, in Appendix \ref{Sec:Appendix}, we give the details on deriving the reformulation \eqref{PNPF-2} of the original PNPF system \eqref{PNPF-1} and the perturbed equations \eqref{PNPF-Perturbed}.


\section{Linearized dissipative laws and the equilibria}\label{Sec:Linear-Dissp}

In this section, we aim at studying the dissipative structures of the linearized equations of the system \eqref{PNPF-Perturbed} near some proper constant equilibria $(\delta_1, \delta_2, \cdots, \delta_N)$ associated with the local density distributions $\rho_i$ $(i = 1, 2, \cdots, N)$. More precisely, the linearized system of \eqref{PNPF-Perturbed} reads
\begin{equation}\label{PNPF-Linearized}
  \left\{
    \begin{aligned}
      \partial_t n_i - \tfrac{k_B}{\nu_i} \Delta n_i = & \tfrac{k_B \delta_i}{\nu_i} \Delta \theta + \tfrac{ z_i \delta_i}{\nu_i} \Delta \phi \,, \ i= 1, 2, \cdots, N \,, \\
      \qquad \qquad \qquad - \Delta \phi = & \tfrac{1}{\eps} m \,, \\
      a \partial_t \theta - b \Delta \theta = & \sum_{i=1}^N \tfrac{k_B^2}{\nu_i} \Delta n_i + \sum_{i=1}^N \tfrac{ k_B z_i \delta_i}{\nu_i} \Delta \phi \,, \\
      \partial_t m - \tfrac{k_B}{\nu} \Delta m + \tfrac{1}{\eps} \big( \sum_{i=1}^N \tfrac{z_i^2 \delta_i}{\nu_i} \big) m = & k_B \sum_{i=1}^N ( \tfrac{1}{\nu_i} - \tfrac{1}{\nu} ) z_i \Delta n_i + k_B \sum_{i=1}^N \tfrac{z_i \delta_i}{\nu_i} \Delta \theta \,,
    \end{aligned}
  \right.
\end{equation}
in which the positive constants $\delta_1$, $\delta_2$, $\cdots$, $\delta_N$ with the constraint $\sum_{j=1}^N \delta_j z_j = 0$ is to be determined, and the constants $a$, $b$ are defined in \eqref{Coeffs-ab}. One easily observes that if $\nu_1 = \nu_2 = \cdots = \nu_N > 0$, the term $k_B \sum_{i=1}^N ( \tfrac{1}{\nu_i} - \tfrac{1}{\nu} ) z_i \Delta n_i$ will vanish. We note that the incompressible solvent velocity $\u_0$ does not affect the linear part of the evolutions. In the following, we will find some admissible equilibria $\delta_1$, $\delta_2$, $\cdots$, $\delta_N$ such that the linearized system \eqref{PNPF-Linearized} is dissipative.

We first derive the energy bounds on the Poisson equation $\phi$ by employing the linearized evolution of the total electrical charge $m$ in \eqref{PNPF-Linearized}. Notice that
\begin{equation*}
  \begin{aligned}
    \partial_t \big( - \eps \Delta \phi \big) + \tfrac{k_B \eps}{\nu} \Delta^2 \phi - \big( \sum_{i=1}^N \tfrac{z_i^2 \delta_i}{\nu_i} \big) \Delta \phi = k_B \sum_{i=1}^N ( \tfrac{1}{\nu_i} - \tfrac{1}{\nu} ) z_i \Delta n_i + k_B \sum_{i=1}^N \tfrac{z_i \delta_i}{\nu_i} \Delta \theta \,,
  \end{aligned}
\end{equation*}
which derives from multiplying by $\phi$ and integrating by parts over $x \in \R^3$ that
  \begin{align*}
    & \tfrac{1}{2} \tfrac{\d}{\d t} \big( \eps \| \nabla \phi \|^2_{L^2} \big) + \tfrac{k_B \eps}{\nu } \| \Delta \phi \|^2_{L^2} + \big( \sum_{i=1}^N \tfrac{z_i^2 \delta_i}{\nu_i} \big) \| \nabla \phi \|^2_{L^2} \\
    = & - k_B \sum_{i=1}^N ( \tfrac{1}{\nu_i} - \tfrac{1}{\nu} ) z_i \langle \nabla n_i , \nabla \phi \rangle - k_B \sum_{i=1}^N \tfrac{z_i \delta_i}{\nu_i} \langle \nabla \theta , \nabla \phi \rangle \\
    \leq & \sum_{i=1}^N \tfrac{k_B^2 (\frac{1}{\nu_i} - \frac{1}{\nu})^2 z_i^2}{4 \eta_{\phi i}} \| \nabla n_i \|^2_{L^2} + \sum_{i=1}^N \eta_{\phi i} \| \nabla \phi \|^2_{L^2} + \tfrac{k_B^2}{4 \eta_\phi} \big( \sum_{i=1}^N \tfrac{z_i \delta_i}{\nu_i} \big)^2 \| \nabla \theta \|^2_{L^2} + \eta_\phi \| \nabla \phi \|^2_{L^2}
  \end{align*}
for some positive constants $\eta_{\phi i}, \eta_\phi > 0$ ($i = 1, 2, \cdots, N$) to be determined. We therefore have
\begin{equation}\label{Lnrzd-phi}
  \begin{aligned}
    \tfrac{1}{2} \tfrac{\d}{\d t} \big( \eps \| \nabla \phi \|^2_{L^2} \big) + \tfrac{k_B \eps}{\nu } \| \Delta \phi \|^2_{L^2} + \big( \sum_{i=1}^N \tfrac{z_i^2 \delta_i}{\nu_i} - \sum_{i=1}^N \eta_{\phi i} - \eta_\phi \big) \| \nabla \phi \|^2_{L^2} \\
    \leq \sum_{i=1}^N \tfrac{k_B^2 (\frac{1}{\nu_i} - \frac{1}{\nu})^2 z_i^2}{4 \eta_{\phi i}} \| \nabla n_i \|^2_{L^2} + \tfrac{k_B^2}{4 \eta_\phi} \big( \sum_{i=1}^N \tfrac{z_i \delta_i}{\nu_i} \big)^2 \| \nabla \theta \|^2_{L^2} \,.
  \end{aligned}
\end{equation}
We next take $L^2$-inner product in the $n_i$-equation of \eqref{PNPF-Linearized} by dot with $n_i$. We then have
\begin{equation*}
  \begin{aligned}
    \tfrac{1}{2} \tfrac{\d}{\d t} \| n_i \|^2_{L^2} + & \tfrac{k_B}{\nu_i} \| \nabla n_i \|^2_{L^2} = - \tfrac{k_B \delta_i}{\nu_i} \l \nabla \theta, \nabla n_i \r - \tfrac{z_i \delta_i}{\nu_i} \l \nabla \phi, \nabla n_i \r \\
    \leq & \eta_i \tfrac{k_B}{\nu_i} \| \nabla n_i \|^2_{L^2} + \tfrac{k_B \delta_i^2}{4 \eta_i \nu_i} \| \nabla \theta \|^2_{L^2} + \eta_i' \tfrac{k_B}{\nu_i} \| \nabla n_i \|^2_{L^2} + \tfrac{z_i^2 \delta_i^2}{4 \eta_i' k_B \nu_i} \| \nabla \phi \|^2_{L^2}
  \end{aligned}
\end{equation*}
for some positive constants $\eta_i, \eta_i' > 0$ ($i = 1, 2, \cdots, N$) to be determined, where the last inequality is derived from the H\"older inequality and the Young's inequality. We thereby obtain
\begin{equation}\label{L2-ni-L}
  \begin{aligned}
    \tfrac{1}{2} \tfrac{\d}{\d t} \| n_i \|^2_{L^2} + ( 1 - \eta_i - \eta_i' ) \tfrac{k_B}{\nu_i} \| \nabla n_i \|^2_{L^2} \leq \tfrac{k_B \delta_i^2}{4 \eta_i \nu_i} \| \nabla \theta \|^2_{L^2} + \tfrac{z_i^2 \delta_i^2}{4 \eta_i' k_B \nu_i} \| \nabla \phi \|^2_{L^2}
  \end{aligned}
\end{equation}
for $i = 1, 2, \cdots, N$. From the same arguments of the inequality \eqref{L2-ni-L}, we can deduce that
\begin{equation}\label{L2-theta-L}
  \begin{aligned}
    \tfrac{1}{2} \tfrac{\d}{\d t} ( a \| \theta \|^2_{L^2} ) + (1 - \eta_\theta - \eta_\theta' ) b \| \nabla \theta \|^2_{L^2} \leq \sum_{i=1}^N \tfrac{N k_B^4}{4 \eta_\theta b \nu_i} \| \nabla n_i \|^2_{L^2} + \tfrac{k_B^2}{4 \eta_\theta' b} \Big( \sum_{i=1}^N \tfrac{z_i \delta_i}{\nu_i} \Big)^2 \| \nabla \phi \|^2_{L^2}
  \end{aligned}
\end{equation}
and
\begin{equation}\label{L2-m-L}
  \begin{aligned}
    \tfrac{1}{2} \tfrac{\d}{\d t} \| m \|^2_{L^2} + (1 - \eta_m - \eta_m') \tfrac{k_B}{\nu} \| \nabla m \|^2_{L^2} + \tfrac{1}{\eps} \sum_{i=1}^N \tfrac{z_i^2 \delta_i}{\nu_i} \| m \|^2_{L^2} \\
    \leq \sum_{i=1}^N \tfrac{k_B \nu z_i^2}{4 \eta_m} ( \tfrac{1}{\nu_i} - \tfrac{1}{\nu} )^2 \| \nabla n_i \|^2_{L^2} + \tfrac{k_B \nu}{4 \eta_m'} \Big( \sum_{i=1}^N \tfrac{z_i \delta_i}{\nu_i} \Big)^2 \| \nabla \theta \|^2_{L^2} \,,
  \end{aligned}
\end{equation}
where the positive constants $\eta_\theta$, $\eta_\theta'$, $\eta_m$ and $\eta_m'$ are to be determined. We now add the $\chi_\phi$ times of \eqref{Lnrzd-phi}, $\chi_i$ times of \eqref{L2-ni-L} and $\chi_m$ times of \eqref{L2-m-L} into $\chi_\theta $ times of the inequality \eqref{L2-theta-L}, where the constants $\chi_\phi, \chi_i, \chi_m, \chi_\theta > 0$ $(i = 1, 2, \cdots, N)$ are also to be determined. We therefore obtain
  \begin{align}\label{L2-Summary-L}
    \no & \tfrac{1}{2} \tfrac{\d}{\d t} \Big( \chi_\phi \eps \| \nabla \phi \|^2_{L^2} + \sum_{i=1}^N \chi_i \| n_i \|^2_{L^2} + a \chi_\theta \| \theta \|^2_{L^2} + \chi_m\| m \|^2_{L^2} \Big) \\
    \no & + \Big[ \chi_\theta (1 - \eta_\theta - \eta_\theta') b - \chi_m \tfrac{k_B \nu}{4 \eta_m'} \big( \sum_{i=1}^N \tfrac{z_i \delta_i}{\nu_i} \big)^2 - \sum_{i=1}^N \chi_i \tfrac{k_B \delta_i^2}{4 \eta_i \nu_i} - \chi_\phi \tfrac{k_B^2}{4 \eta_\phi} \big( \sum_{i=1}^N \tfrac{z_i \delta_i}{\nu_i} \big)^2 \Big] \| \nabla \theta \|^2_{L^2} \\
    \no & + \chi_m (1 - \eta_m - \eta_m' ) \tfrac{k_B}{\nu} \| \nabla m \|^2_{L^2} + \chi_m \tfrac{1}{\eps} \sum_{i=1}^N \tfrac{z_i^2 \delta_i}{\nu_i} \| m \|^2_{L^2} + \chi_\phi \tfrac{k_B \eps}{\nu } \| \Delta \phi \|^2_{L^2} \\
    & + \sum_{i=1}^N \big[ \chi_i ( 1 - \eta_i - \eta_i' ) \tfrac{k_B}{\nu_i} - \chi_\theta \tfrac{N k_B^4}{4 \eta_\theta b \nu_i} - \chi_m \tfrac{k_B \nu z_i^2}{4 \eta_m} ( \tfrac{1}{\nu_i} - \tfrac{1}{\nu} )^2 - \chi_\phi \tfrac{k_B^2 z_i^2}{4 \eta_{\phi i}} ( \tfrac{1}{\nu_i} - \tfrac{1}{\nu} )^2 \big] \| \nabla n_i \|^2_{L^2} \\
    \no & + \Big[ \chi_\phi \big( \sum_{i=1}^N \tfrac{z_i^2 \delta_i}{\nu_i} - \sum_{i=1}^N \eta_{\phi i} - \eta_\phi \big) - \chi_\theta \tfrac{k_B^2}{4 \eta_\theta' b} \big( \sum_{i=1}^N \tfrac{z_i \delta_i}{\nu_i} \big)^2 - \sum_{i=1}^N \chi_i \tfrac{z_i^2 \delta_i^2}{4 \eta_i' k_B \nu_i} \Big] \| \nabla \phi \|^2_{L^2} \leq 0 \,.
  \end{align}
Then the basic energy law \eqref{L2-Summary-L} is dissipative if and only if there are some positive constants
\begin{equation*}
  \chi_i \,, \ \eta_i \,, \ \eta_i' \,, \ \eta_{\phi i} \, (1 \leq i \leq N) \,, \ \chi_m \,, \ \eta_m \,, \ \eta_m' \,, \ \eta_\theta \,, \ \eta_\theta' \,, \ \eta_\phi \,, \ \chi_\phi \,, \ \chi_\theta \,,
\end{equation*}
which may depend on the choices of $\delta_i$ $(i= 1, 2, \cdots, N)$, such that
\begin{align*}
  \textrm{(H1): } & \ \chi_\theta (1 - \eta_\theta - \eta_\theta') b - \chi_m \tfrac{k_B \nu}{4 \eta_m'} \big( \sum_{i=1}^N \tfrac{z_i \delta_i}{\nu_i} \big)^2 - \sum_{i=1}^N \chi_i \tfrac{k_B \delta_i^2}{4 \eta_i \nu_i} - \chi_\phi \tfrac{k_B^2}{4 \eta_\phi} \big( \sum_{i=1}^N \tfrac{z_i \delta_i}{\nu_i} \big)^2 > 0 \,, \\
  \textrm{(H2): } & \ 1 - \eta_m - \eta_m' > 0 \,, \\
  \textrm{(H3): } & \ \chi_\phi \big( \sum_{i=1}^N \tfrac{ z_i^2 \delta_i}{\nu_i} - \sum_{i=1}^N \eta_{\phi i} - \eta_\phi \big) - \chi_\theta \tfrac{ k_B^2}{4 \eta_\theta' b} \big( \sum_{i=1}^N \tfrac{z_i \delta_i}{\nu_i} \big)^2 - \sum_{i=1}^N \chi_i \tfrac{z_i^2 \delta_i^2}{4 \eta_i' k_B \nu_i} > 0 \,, \\
  \textrm{(H4): } & \ \chi_i ( 1 - \eta_i - \eta_i' ) \tfrac{k_B}{\nu_i} - \chi_\theta \tfrac{N k_B^4}{4 \eta_\theta b \nu_i} - \chi_m \tfrac{k_B \nu z_i^2}{4 \eta_m} ( \tfrac{1}{\nu_i} - \tfrac{1}{\nu} )^2 - \chi_\phi \tfrac{k_B^2 z_i^2}{4 \eta_{\phi i}} ( \tfrac{1}{\nu_i} - \tfrac{1}{\nu} )^2 > 0 \,, \ 1 \leq i \leq N \,,
\end{align*}
where $b > 0$ is defined in \eqref{Coeffs-ab}.

We now introduce a set $\mathcal{S}_{eq}$ of the elements $(\delta_1, \delta_2, \cdots, \delta_N)$, which contains all possible equilibria $\delta_1, \delta_2, \cdots, \delta_N$ associated with the local density distributions $\rho_1$, $\rho_2$, $\cdots$, $\rho_N$, respectively, such that the linearized system \eqref{PNPF-Linearized} is dissipative.

\begin{definition}[Equilibria set $\mathcal{S}_{eq}$]\label{Def-AEF}
	We define a equilibria set $\mathcal{S}_{eq} \subseteq \R^N$, whose elements $(\delta_1, \delta_2, \cdots, \delta_N)$ satisfy the following two conditions:
	\begin{enumerate}
		\item[\bf (A1)] $\delta_i > 0$ for $i = 1, 2, \cdots, N$ and $\sum_{i=1}^N z_i \delta_i = 0$;
		\item[\bf (A2)] There are some positive constants
		\begin{equation*}
		  \chi_i \,, \ \eta_i \,, \ \eta_i' \,, \ \eta_{\phi i} \, (1 \leq i \leq N) \,, \ \chi_m \,, \ \eta_m \,, \ \eta_m' \,, \ \eta_\theta \,, \ \eta_\theta' \,, \ \eta_\phi \,, \ \chi_\phi \,, \ \chi_\theta \,,
		\end{equation*}
		which may depend on the choices of $\delta_i$ $(i= 1, 2, \cdots, N)$, such that the hypotheses $\mathrm{(H1)}$-$\mathrm{(H4)}$ hold.
	\end{enumerate}
\end{definition}

We remark that the equilibria set $\mathcal{S}_{eq}$ depends only on the all coefficients and the species number $N \geq 2$.

Next, for the equilibria set $\mathcal{S}_{eq}$ defined in Definition \ref{Def-AEF}, we introduce the following proposition to prove the set $\mathcal{S}_{eq}$ is nonempty. Once the following proposition holds, the basic energy law \eqref{L2-Summary-L} is dissipative with any fixed equilibrium belonging to $\mathcal{S}_{eq}$ associated with the local density distributions $\rho_1, \rho_2, \cdots, \rho_N$.

\begin{proposition}\label{Prop-AEF}
	Under the assumptions \eqref{Asmp-z} and \eqref{Asmp-othCoefs}, we further assume \eqref{Key-Asmp}, i.e.,
	\begin{equation*}
	  \begin{aligned}
	    \max_{1 \leq i \leq N} ( 1 - \tfrac{\nu_i}{\nu} )^2 < \tfrac{1}{2} \,.
	  \end{aligned}
	\end{equation*}
	Then we have
	\begin{equation}
	  \begin{aligned}
	    \mathcal{S}_{eq} \neq \emptyset \,.
	  \end{aligned}
	\end{equation}
\end{proposition}

\begin{proof}
	Let $\eta_m = \eta_m' = \tfrac{1}{4}$. Then the hypothesis $\textrm{(H2)}$ automatically holds. We will prove the hypotheses $\textrm{(H1)}$, $\textrm{(H3)}$ and $\textrm{(H4)}$ by three steps.
	
	{\bf Claim 1.} {\em There are $\eps_0 \in (0, 1)$ $($closed to $1)$ and $\lambda \in (\tfrac{1}{2 \eps_0}, 1)$ such that for any fixed $\kappa > 0$, if $(\delta_1, \delta_2, \cdots, \delta_N) \in \R^N$ satisfy
	\begin{equation}\label{Clm1-0}
	  \begin{aligned}
	    (\tfrac{1}{2 \lambda \eps_0})^\frac{1}{\kappa} < \delta_i < \Big[ \min \{ 1, \tfrac{(1 - \lambda ) \eps_0 }{(1 - \frac{\nu_i}{\nu})^2} \} \Big]^\frac{1}{\kappa} \ (i = 1, 2, \cdots, N)
	  \end{aligned}
	\end{equation}
	with $\sum_{i=1}^N z_i \delta_i = 0$, then there exist numbers $\eta_{\phi i}, \eta_\phi, \eta_i, \eta_i', \chi_i , \chi_\phi > 0$ $(1 \leq i \leq N)$ satisfying
	\begin{equation}\label{Clm1}
	  \left\{
	    \begin{aligned}
	      & \chi_\phi \big( \sum_{i=1}^N \tfrac{ z_i^2 \delta_i}{\nu_i} - \sum_{i=1}^N \eta_{\phi i} - \eta_\phi \big) - \sum_{i=1}^N \chi_i \tfrac{z_i^2 \delta_i^2}{4 \eta_i' k_B \nu_i} > 0 \,, \\
	      & \chi_i ( 1 - \eta_i - \eta_i' ) \tfrac{k_B}{\nu_i} - \chi_\phi \tfrac{k_B^2 z_i^2}{4 \eta_{\phi i}} ( \tfrac{1}{\nu_i} - \tfrac{1}{\nu} )^2 > 0 \,, \ 1 \leq i \leq N \,.
	    \end{aligned}
	  \right.
	\end{equation}}
	
	Indeed, we first set $\eta_i = 1 - \eps_0 > 0$, $\eta_{\phi i} = \tfrac{ z_i^2 \delta_i}{4 \nu_i}$ for $1 \leq i \leq N$, $\eta_\phi = \sum_{i=1}^N \tfrac{ z_i^2 \delta_i}{4 \nu_i} $ and $\eta_i' = \tfrac{1}{2} \delta_i^{- \kappa}$ ($1 \leq i \leq N$) for any fixed $\kappa > 0$. Here $0 < \eps_0 < 1$ is to be determined. It therefore sees that
	\begin{equation}\label{Clm1-1}
	  \begin{aligned}
	    \sum_{i=1}^N \tfrac{ z_i^2 \delta_i}{\nu_i} - \sum_{i=1}^N \eta_{\phi i} - \eta_\phi = \sum_{i=1}^N \tfrac{ z_i^2 \delta_i}{2 \nu_i} > 0
	  \end{aligned}
	\end{equation}
	and
	\begin{equation*}
	  \begin{aligned}
	    1 - \eta_i - \eta_i' = \eps_0 - \tfrac{1}{2} \delta_i^{- \kappa} \,.
	  \end{aligned}
	\end{equation*}
	We first require $\lambda \eps_0 - \tfrac{1}{2} \delta_i^{- \kappa} > 0$ for some $0 < \lambda < 1$. We then see that $\delta_i^\kappa > \tfrac{1}{2 \lambda \eps_0} > 0$. We further restrict $\tfrac{1}{2 \lambda \eps_0} < 1$, so that
	\begin{equation*}
	  \begin{aligned}
	    \mathcal{M}_1^\kappa : = \Big\{ (\delta_1, \delta_2, \cdots, \delta_N) \in \R^N ; (\tfrac{1}{2 \lambda \eps_0})^\frac{1}{\kappa} < \delta_i < 1 \,, \ i = 1, 2, \cdots, N \,, \ \sum_{i=1}^N z_i \delta_i = 0 \Big\} \neq \emptyset \,.
	  \end{aligned}
	\end{equation*}
	Consequently, all elements in $\mathcal{M}_1^\kappa$ are such that
	\begin{equation}\label{Clm1-2}
	  \begin{aligned}
	    1 - \eta_i - \eta_i' = (1 - \lambda ) \eps_0 + \lambda \eps_0 - \tfrac{1}{2} \delta_i^{- \kappa} > (1 - \lambda ) \eps_0 > 0 \,, \ i = 1, 2, \cdots, N  \,.
	  \end{aligned}
	\end{equation}
	Together with \eqref{Clm1-1}, the first inequality of \eqref{Clm1} can be rewritten as
	\begin{equation*}
	  \begin{aligned}
	    \chi_\phi \sum_{i=1}^N \tfrac{ z_i^2 \delta_i}{2 \nu_i} > \sum_{i=1}^N \chi_i \tfrac{z_i^2 \delta_i^{2 + \kappa}}{2 k_B \nu_i} \,,
	  \end{aligned}
	\end{equation*}
	which will hold provided that $\chi_\phi \tfrac{ z_i^2 \delta_i}{2 \nu_i} > \chi_i \tfrac{z_i^2 \delta_i^{2 + \kappa}}{2 k_B \nu_i}$, namely,
	\begin{equation}\label{Clm1-3}
	  \begin{aligned}
	    \chi_\phi > \chi_i \tfrac{\delta_i^{1 + \kappa}}{k_B} \,, \ 1 \leq i \leq N \,.
	  \end{aligned}
	\end{equation}
	Combining with the choice of $\eta_{\phi i}$ and \eqref{Clm1-2}, we transform the second inequality \eqref{Clm1} into
	\begin{equation}\label{Clm1-4}
	  \begin{aligned}
	    \chi_i > \chi_\phi \tfrac{k_B}{ \delta_i (\eps_0 - \frac{1}{2} \delta_i^{- \kappa})} (1 - \tfrac{\nu_i}{\nu})^2 \,, \ 1 \leq i \leq N \,.
	  \end{aligned}
	\end{equation}
	Our goal is to prove that the inequalities \eqref{Clm1-3} and \eqref{Clm1-4} will hold for some $\chi_i, \chi_\phi > 0$, hence,
	\begin{equation}\label{Clm1-5}
	  \begin{aligned}
	    1 > \tfrac{\delta_i^{1+\kappa}}{k_B} \tfrac{\chi_i}{\chi_\phi} > \tfrac{\delta_i^\kappa ( 1 - \tfrac{\nu_i}{\nu} )^2 }{( \eps_0 - \tfrac{1}{2} \delta_i^{- \kappa} )} \,, \ 1 \leq i \leq N \,.
	  \end{aligned}
	\end{equation}
	Noticing the bound \eqref{Clm1-2}, to make the \eqref{Clm1-5} established for some $\chi_i, \chi_\phi > 0$, we only need to choose $(\delta_1, \delta_2, \cdots, \delta_N)$ in the following set:
	\begin{equation*}
	  \begin{aligned}
	    \mathcal{M}_2^\kappa : = \Big\{ (\delta_1, \delta_2, \cdots, \delta_N) \in \R^N ; 0 < \delta_i^\kappa < \tfrac{1}{ (1 - \tfrac{\nu_i}{\nu})^2 } (1- \lambda) \eps_0 \,, \ 1 \leq i \leq N \Big\} \,.
	  \end{aligned}
	\end{equation*}
	Consequently, we shall prove that
	\begin{equation*}
	  \begin{aligned}
	    \mathcal{M}_1^\kappa \cap \mathcal{M}_2^\kappa \neq \emptyset \,,
	  \end{aligned}
	\end{equation*}
	which is equivalent to show that there are $\eps_0 \in (0,1)$ and $\lambda \in (\tfrac{1}{2 \eps_0}, 1)$ such that
	\begin{equation}\label{Clm1-6}
	  \begin{aligned}
	    \tfrac{1}{2 \lambda \eps_0} < \tfrac{1}{ (1 - \tfrac{\nu_i}{\nu})^2 } (1- \lambda) \eps_0 \,, \ \forall \, 1 \leq i \leq N \,.
	  \end{aligned}
	\end{equation}
	
	Let $\lambda = \tfrac{1}{2 \eps_0} + \gamma$ for some $\gamma > 0$ to be determined. Then, we see
	\begin{equation}\label{Clm1-7}
	  \begin{aligned}
	    \lambda (1 - \lambda) \eps_0^2 = (\tfrac{1}{2 \eps_0} + \gamma) (1 - \tfrac{1}{2 \eps_0} - \gamma) \eps_0^2 < \tfrac{1}{4}
	  \end{aligned}
	\end{equation}
	and
	\begin{equation}\label{Clm1-8}
	  \begin{aligned}
	    \lim_{\substack{\eps_0 \rightarrow 1- \\ \gamma \rightarrow 0+ }} (\tfrac{1}{2 \eps_0} + \gamma) (1 - \tfrac{1}{2 \eps_0} - \gamma) \eps_0^2 = \tfrac{1}{4} \,.
	  \end{aligned}
	\end{equation}
	Furthermore, the condition \eqref{Key-Asmp} tells us that
	\begin{equation}\label{Clm1-9}
	  \begin{aligned}
	    \tfrac{(1 - \tfrac{\nu_i}{\nu})^2}{2} < \tfrac{1}{4} \, (\forall \, 1 \leq i \leq N) \,.
	  \end{aligned}
	\end{equation}
	We thereby conclude \eqref{Clm1-6} from \eqref{Clm1-7}, \eqref{Clm1-8} and \eqref{Clm1-9}. Namely, we have $\mathcal{M}_1^\kappa \cap \mathcal{M}_2^\kappa \neq \emptyset$ and Claim 1 holds.
	
	{\bf Claim 2.} {\em There are $\delta_1, \delta_2, \cdots , \delta_N > 0$ with $\sum_{i=1}^N z_i \delta_i = 0$ such that
	\begin{equation}\label{Cm2}
	  \left\{
	    \begin{aligned}
	      & \chi_\theta (1 - \eta_\theta - \eta_\theta') b - \sum_{i=1}^N \chi_i \tfrac{k_B \delta_i^2}{4 \eta_i \nu_i} - \chi_\phi \tfrac{k_B^2}{4 \eta_\phi} \big( \sum_{i=1}^N \tfrac{z_i \delta_i}{\nu_i} \big)^2 > 0 \,, \\
	      & \chi_\phi \big( \sum_{i=1}^N \tfrac{ z_i^2 \delta_i}{\nu_i} - \sum_{i=1}^N \eta_{\phi i} - \eta_\phi \big) - \chi_\theta \tfrac{ k_B^2}{4 \eta_\theta' b} \big( \sum_{i=1}^N \tfrac{z_i \delta_i}{\nu_i} \big)^2 - \sum_{i=1}^N \chi_i \tfrac{z_i^2 \delta_i^2}{4 \eta_i' k_B \nu_i} > 0 \,, \\
	      & \chi_i ( 1 - \eta_i - \eta_i' ) \tfrac{k_B}{\nu_i} - \chi_\theta \tfrac{N k_B^4}{4 \eta_\theta b \nu_i} - \chi_\phi \tfrac{k_B^2 z_i^2}{4 \eta_{\phi i}} ( \tfrac{1}{\nu_i} - \tfrac{1}{\nu} )^2 > 0 \,, \ 1 \leq i \leq N \,,
	    \end{aligned}
	  \right.
	\end{equation}
	hold for some positive constants $\chi_i \,, \ \eta_i \,, \ \eta_i' \,, \ \eta_{\phi i} \, (1 \leq i \leq N) \,, \ \eta_\theta \,, \ \eta_\theta' \,, \ \eta_\phi \,, \ \chi_\phi \,, \ \chi_\theta$, depending on $\delta_i (1 \leq i \leq N)$ and the all coefficients.}

    Indeed, we first choose $\chi_i \,, \ \eta_i \,, \ \eta_i' \,, \ \eta_{\phi i} \, (1 \leq i \leq N) \,, \ \eta_\phi \,, \ \chi_\phi $ as the same as in the process of proving Claim 1. We further take $\eta_\theta = \eta_\theta' = \tfrac{1}{3}$. Then the inequalities \eqref{Cm2} can be transformed into
    \begin{equation}\label{Cm2-1}
      \left\{
        \begin{aligned}
          & \tfrac{\chi_\theta}{\chi_\phi} > \sum_{i=1}^N \tfrac{\chi_i}{\chi_\phi} \tfrac{k_B}{4 b \nu_i (1 - \eps_0)} \delta_i^2 + \tfrac{k_B^2}{b} \tfrac{\big( \sum_{i=1}^N \tfrac{z_i \delta_i}{\nu_i} \big)^2}{ \sum_{i=1}^N \tfrac{z_i^2 \delta_i}{\nu_i} } \,, \\
          & \tfrac{\chi_\theta}{\chi_\phi} \big( \sum_{i=1}^N \tfrac{z_i \delta_i}{\nu_i} \big)^2 < \sum_{i=1}^N \tfrac{2b z_i^2}{3 k_B^2 \nu_i} \delta_i - \sum_{i=1}^N \tfrac{\chi_i}{\chi_\phi} \tfrac{2b z_i^2}{3 k_B^3 \nu_i} \delta_i^{2+\kappa} \,, \\
          & \tfrac{\chi_\theta}{\chi_\phi} < \tfrac{\chi_i}{\chi_\phi} \tfrac{4b (\eps_0 - \tfrac{1}{2} \delta_i^{-\kappa})}{3 N k_B^3} - \tfrac{4 b ( 1 - \tfrac{\nu_i}{\nu} )^2 }{3 N k_B^2} \delta_i^{-1} \,, \ 1 \leq i \leq N \,,
        \end{aligned}
      \right.
    \end{equation}
    where, from \eqref{Clm1-5}, $\tfrac{\chi_i}{\chi_\phi} (1 \leq i \leq N)$ satisfy
    \begin{equation}\label{Cm2-2}
      \begin{aligned}
        \tfrac{k_B (1 - \tfrac{\nu_i}{\nu})^2}{ \eps_0 - \tfrac{1}{2} \delta_i^{- \kappa} } \delta_i^{-1} < \tfrac{\chi_i}{\chi_\phi} < k_B \delta_i^{-1-\kappa} \,.
      \end{aligned}
    \end{equation}

    In order to prove \eqref{Cm2-1}, by the consideration of the Cauchy inequalities
    \begin{equation*}
      \begin{aligned}
        \big( \sum_{i=1}^N \tfrac{z_i \delta_i}{\nu_i} \big)^2 \leq \big( \sum_{i=1}^N \tfrac{z_i^2 \delta_i}{\nu_i} \big) \big( \sum_{i=1}^N \tfrac{\delta_i}{\nu_i} \big) \,, \quad \big( \sum_{i=1}^N \tfrac{z_i \delta_i}{\nu_i} \big)^2 \leq N \sum_{i=1}^N \tfrac{z_i \delta_i^2}{\nu_i^2} \,,
      \end{aligned}
    \end{equation*}
    it suffices to prove
    \begin{equation}\label{Cm2-3}
      \begin{aligned}
        \Big( \tfrac{N k_B}{4 b \nu_i (1 - \eps_0)} \tfrac{\chi_i}{\chi_\phi} \delta_i + \tfrac{N k_B^2}{b \nu_i} \Big) \delta_i < \tfrac{\chi_\theta}{\chi_\phi} < \min \{ \mathcal{A}_1 (\tfrac{\chi_i}{\chi_\phi}), \mathcal{A}_2 (\tfrac{\chi_i}{\chi_\phi}) \} \delta_i^{-1} \,,
      \end{aligned}
    \end{equation}
    where
    \begin{equation*}
      \begin{aligned}
        \mathcal{A}_1 (\tfrac{\chi_i}{\chi_\phi}) = \tfrac{2b}{3 N k_B^2} -\tfrac{2 b \delta_i^\kappa}{3 N k_B^3} \tfrac{\chi_i}{\chi_\phi} \delta_i \,, \quad \mathcal{A}_2 (\tfrac{\chi_i}{\chi_\phi}) = \tfrac{4b (\eps_0 - \tfrac{1}{2} \delta_i^{-\kappa})}{3 N k_B^3} \tfrac{\chi_i}{\chi_\phi} \delta_i - \tfrac{4 b (1 - \tfrac{\nu_i}{\nu})^2}{3 N k_B^2} \,.
      \end{aligned}
    \end{equation*}
    It is easy to see that $\mathcal{A}_1 (\tfrac{\chi_i}{\chi_\phi}), \mathcal{A}_2 (\tfrac{\chi_i}{\chi_\phi}) > 0$ for all $\tfrac{\chi_i}{\chi_\phi}$ given in \eqref{Cm2-2}. Without loss of generality, we fix
    \begin{equation*}
      \begin{aligned}
        \tfrac{\chi_i}{\chi_\phi} = \tfrac{1}{2} \Big[ \tfrac{k_B (1 - \tfrac{\nu_i}{\nu})^2}{ \eps_0 - \tfrac{1}{2} \delta_i^{- \kappa} } + \tfrac{k_B}{\delta_i^\kappa} \Big] \delta_i^{-1} \,.
      \end{aligned}
    \end{equation*}
	From \eqref{Clm1-0} and \eqref{Clm1-2}, we see that there are positive constants $C_1, C_2, > 0$, independent of $\kappa$ and $\delta_i (i=1,2, \cdots, N)$, such that
	\begin{equation*}
	  \begin{aligned}
	    0 < C_1 \leq \min \{ \mathcal{A}_1 (\tfrac{\chi_i}{\chi_\phi}), \mathcal{A}_2 (\tfrac{\chi_i}{\chi_\phi}) \} \leq C_2 \,.
	  \end{aligned}
	\end{equation*}
	By \eqref{Clm1-0} and \eqref{Clm1-2}, we notice that
	\begin{equation*}
	  \begin{aligned}
	    \tfrac{N k_B}{4 b \nu_i (1 - \eps_0)} \tfrac{\chi_i}{\chi_\phi} \delta_i + \tfrac{N k_B^2}{b \nu_i} \leq \tfrac{N k_B^2}{8 k \nu_i (1 - \eps_0)} \big[ \tfrac{(1 - \tfrac{\nu_i}{\nu})^2}{ \eps_0 (1 - \lambda)} + 2 \lambda \eps_0 \big] + \tfrac{N k_B^2}{k \nu_i} : = C_3 \,.
	  \end{aligned}
	\end{equation*}
	Therefore, in order to prove \eqref{Cm2-3}, it suffices to prove that there are some $\delta_i (1 \leq i \leq N)$ in \eqref{Clm1-0} such that
	\begin{equation*}
	  \begin{aligned}
	    C_3 \delta_i < \tfrac{\chi_\theta}{\chi_\phi} < C_1 \delta_i^{-1} \,, \ i = 1, 2, \cdots, N \,.
	  \end{aligned}
	\end{equation*}
	
	Since
	\begin{equation*}
	  \begin{aligned}
	    \lim_{\kappa \rightarrow 0+} (\tfrac{1}{2 \lambda \eps_0})^\frac{1}{\kappa} = 0 \ \textrm{ and } \ (\tfrac{1}{2 \lambda \eps_0})^\frac{1}{\kappa} > 0 \,,
	  \end{aligned}
	\end{equation*}
	there is a $\kappa_0 > 0$ such that for all $0 < \kappa \leq \kappa_0$,
	\begin{equation*}
	  \begin{aligned}
	    0 < ( \tfrac{1}{2 \lambda \eps_0} )^\frac{1}{\kappa} \leq \tfrac{1}{2} \sqrt{\tfrac{C_1}{C_3}} \,.
	  \end{aligned}
	\end{equation*}
	Fix $\kappa \in (0, \kappa_0]$, we take any $\delta_i$ such that
	\begin{equation*}
	  \begin{aligned}
	    (\tfrac{1}{2 \lambda \eps_0})^\frac{1}{\kappa} < \delta_i < \min \Big\{ \sqrt{\tfrac{C_1}{C_3}} ,  \Big[ \min \{ 1, \tfrac{(1 - \lambda ) \eps_0 }{(1 - \frac{\nu_i}{\nu})^2} \}  \Big]^\frac{1}{\kappa} \Big\} \ (i = 1, 2, \cdots, N) \,.
	  \end{aligned}
	\end{equation*}
	These $\delta_i$ are that we want to find, and Claim 2 holds.
	
	At the final step, based on Claim 2, we will close the proof of Proposition \ref{Prop-AEF}. If $\sum_{i=1}^N \tfrac{z_i \delta_i}{\nu_i} = 0$ and $\nu_1 = \cdots = \nu_N = \nu > 0$, $\chi_m$ can be taken as any positive constant. We only need to consider the case $\sum_{i=1}^N \tfrac{z_i \delta_i}{\nu_i} \neq 0$ or $\nu_j \neq \nu$ for some $j$. Since $\eta_m = \eta_m' = \tfrac{1}{4}$, the (H1) and (H4) imply
	\begin{equation*}
	  \begin{aligned}
	    \chi_m & < \tfrac{\chi_\theta (1 - \eta_\theta - \eta_\theta') b - \sum_{i=1}^N \chi_i \tfrac{k_B \delta_i^2}{4 \eta_i \nu_i} - \chi_\phi \tfrac{k_B^2}{4 \eta_\phi} \big( \sum_{i=1}^N \tfrac{z_i \delta_i}{\nu_i} \big)^2}{k_B \nu ( \sum_{i=1}^N \tfrac{z_i \delta_i}{\nu_i} )^2} : = \mathcal{C}_1 \,, \\
	    \chi_m & < \min_{1 \leq i \leq N} \tfrac{\chi_i ( 1 - \eta_i - \eta_i' ) \tfrac{k_B}{\nu_i} - \chi_\theta \tfrac{N k_B^4}{4 \eta_\theta b \nu_i} - \chi_\phi \tfrac{k_B^2 z_i^2}{4 \eta_{\phi i}} ( \tfrac{1}{\nu_i} - \tfrac{1}{\nu} )^2}{k_B \nu z_i^2 (\tfrac{1}{\nu_i} - \tfrac{1}{\nu})^2} : = \mathcal{C}_2 \,.
	  \end{aligned}
	\end{equation*}
	Here Claim 2 guarantees that $\mathcal{C}_1, \mathcal{C}_2 > 0$. Consequently, we take $\chi_m = \tfrac{1}{2} \min \{ \mathcal{C}_1, \mathcal{C}_2 \} > 0$ and the proof of Proposition \ref{Prop-AEF} is finished.
	
\end{proof}

\section{Global well-posedness with small initial data}\label{Sec:Global}

In this section, we will prove the global well-posedness of the $(\rho_1, \rho_2, \cdots, \rho_N, T)$-system \eqref{PNPF-1}-\eqref{IC-PNPF} near the equilibrium $(\delta_1, \delta_2, \cdots, \delta_N, 1)$, where $(\delta_1, \delta_2, \cdots, \delta_N)$ belongs to the equilibria set $\mathcal{S}_{eq}$ given in Definition \ref{Def-AEF}. In this sense, we focus on the perturbed system \eqref{PNPF-Perturbed} with initial data
\begin{equation}\label{IC-perturbed}
  \begin{aligned}
    n_i (0, x) = n_i^\inn (x) \,, \ i = 1, 2, \cdots, N \,, \ \theta (0, x) = \theta^\inn (x) \,,
  \end{aligned}
\end{equation}
where $n_i^\inn (x) = \rho_i^\inn (x) - \delta_i$, $i = 1, 2, \cdots, N$, and $\theta^\inn (x) = T^\inn (x) - 1$.

We employ the mollifier method to construct the approximate solutions:
\begin{equation}\label{Apprx-Syst}
  \left\{
    \begin{array}{l}
      \partial_t n_i^\kappa - \tfrac{k_B}{\nu_i} \Delta \mathcal{J}_\kappa n_i^\kappa = \tfrac{k_B \delta_i}{\nu_i} \Delta \mathcal{J}_\kappa \theta^\kappa + \tfrac{z_i \delta_i}{\nu_i} \Delta \mathcal{J}_\kappa \phi^\kappa + \mathcal{J}_\kappa R_{n_i} ( \mathcal{J}_\kappa n_i^\kappa, \mathcal{J}_\kappa m^\kappa, \phi^\kappa, \u_0^\kappa ) \,, \\[2mm]
      \qquad \qquad \qquad - \Delta \phi^\kappa = \tfrac{1}{\eps} \mathcal{J}_\kappa m^\kappa \,, \\[2mm]
      \lambda_0 \Delta \u_0^\kappa = \nabla P_0^\kappa + \sum_{i=1}^N k_B \mathcal{J}_\kappa \nabla ( n_i^\kappa + \delta_i \theta^\kappa ) \\
      \qquad \qquad \qquad \qquad \qquad \qquad \qquad + \mathcal{J}_\kappa R_{\u_0} ( \mathcal{J}_\kappa n_1^\kappa, \cdots, \mathcal{J}_\kappa n_N^\kappa, \mathcal{J}_\kappa \theta^\kappa, \mathcal{J}_\kappa m^\kappa, \phi^\eps ) \,, \\[2mm]
      \qquad \qquad \qquad \qquad \nabla \cdot \u_0^\kappa = 0 \,, \\[2mm]
      a \partial_t \theta^\kappa - b \Delta \mathcal{J}_\kappa \theta^\kappa = \sum_{i=1}^N \tfrac{k_B^2}{\nu_i} \Delta \mathcal{J}_\kappa n_i^\kappa \\
      \qquad \qquad \qquad  + \sum_{i=1}^N \tfrac{k_B z_i \delta_i}{\nu_i} \Delta \mathcal{J}_\kappa \phi^\kappa + \mathcal{J}_\kappa R_\theta ( \mathcal{J}_\kappa n_1^\kappa, \cdots, \mathcal{J}_\kappa n_N^\kappa, \mathcal{J}_\kappa \theta^\kappa, \phi^\kappa, \u_0^\kappa ) \,, \\[2mm]
      \partial_t m^\kappa - \tfrac{k_B}{\nu} \Delta \mathcal{J}_\kappa m^\kappa + \tfrac{1}{\eps} \big( \sum_{i=1}^N \tfrac{z_i^2 \delta_i}{\nu_i} \big) \mathcal{J}_\kappa m^\kappa = k_B \sum_{i=1}^N ( \tfrac{1}{\nu_i} - \tfrac{1}{\nu} ) z_i \Delta \mathcal{J}_\kappa n_i^\kappa \\
      \qquad \qquad \qquad + k_B \sum_{i=1}^N \tfrac{z_i \delta_i}{\nu_i} \Delta \mathcal{J}_\kappa \theta^\kappa + \mathcal{J}_\kappa R_m ( \mathcal{J}_\kappa n_1^\kappa, \cdots, \mathcal{J}_\kappa n_N^\kappa, \mathcal{J}_\kappa \theta^\kappa, \phi^\kappa, \u_0^\kappa ) \,,
    \end{array}
  \right.
\end{equation}
with initial data
\begin{equation}\label{IC-Apprx-Syst}
  \begin{aligned}
     n_i^\kappa (0, x) = \mathcal{J}_\kappa n_i^\inn (x) \,, \ i = 1, 2, \cdots, N \,, \ \theta^\kappa (0, x) = \mathcal{J}_\kappa \theta^\inn (x) \,.
  \end{aligned}
\end{equation}
It is natural to know that $m^\kappa (0,x) = \sum_{j=1}^N z_j \mathcal{J}_\kappa n_j^\inn (x)$. The mollifier operator $\mathcal{J}_\kappa$ is defined as
$$ \mathcal{J}_\kappa f : = \mathcal{F}^{-1} \Big( \mathbf{1}_{|\xi| \leq \tfrac{1}{\kappa}} ( \mathcal{F} f ) (\xi) \Big) \,,$$
where $\mathcal{F}$ is the standard Fourier transform over the whole space $\R^3$ and $\mathcal{F}^{-1}$ is its inverse transform. Moreover, the mollifier operator $\mathcal{J}_\kappa$ has the property $\mathcal{J}_\kappa^2 = \mathcal{J}_\kappa$.

In the arguments proving the convergence ($\kappa \rightarrow 0$) of the approximate solutions \eqref{Apprx-Syst}-\eqref{IC-Apprx-Syst}, it is essential to obtain uniform (in $\kappa > 0$) energy estimates of \eqref{Apprx-Syst}-\eqref{IC-Apprx-Syst}, whose derivations are the same as the derivations of the a priori estimates for the perturbed system \eqref{PNPF-Perturbed} with the initial data \eqref{IC-perturbed}. The convergence arguments are a standard process. For simplicity, we will only establish a priori estimates for the smooth solutions of \eqref{PNPF-Perturbed}-\eqref{IC-perturbed}. Therefore, let us assume in the rest of this section that $( n_1, \cdots, n_N, \theta, m )$ is a local smooth solution to \eqref{PNPF-Perturbed}-\eqref{IC-perturbed} on some time interval.

We first introduce the following energy functional $\mathscr{E}_s (t)$
\begin{equation}\label{Es}
  \begin{aligned}
    \mathscr{E}_s (t) : = \chi_\phi \eps \| \nabla \phi \|^2_{H^s} + \sum_{i=1}^N \chi_i \| n_i \|^2_{H^s} + a \chi_\theta \| \theta \|^2_{H^s} + \chi_m \| m \|^2_{H^s} \,,
  \end{aligned}
\end{equation}
and the energy dissipation rate functional $\mathscr{D}_s (t)$
\begin{equation}\label{Ds}
  \begin{aligned}
    \mathscr{D}_s (t) : = & \sum_{i=1}^N d_i \| \nabla n_i \|^2_{H^s} + d_\theta \| \nabla \theta \|^2_{H^s} + d_m \| \nabla m \|^2_{H^s} \\
    & + \tilde{d}_m \| m \|^2_{H^s} + d_\phi \| \nabla \phi \|^2_{H^s} + \tilde{d}_\phi \| \Delta \phi \|^2_{H^s} + \| \nabla \u_0 \|^2_{H^s} \,,
  \end{aligned}
\end{equation}
where the constants $d_i (1 \leq i \leq N), d_\theta, d_m, \tilde{d}_m, d_\phi, \tilde{d}_\phi$ are given as
\begin{align*}
  & d_i : = \chi_i ( 1 - \eta_i - \eta_i' ) \tfrac{k_B}{\nu_i} - \chi_\theta \tfrac{N k_B^4}{4 \eta_\theta b \nu_i} - \chi_m \tfrac{k_B \nu z_i^2}{4 \eta_m} ( \tfrac{1}{\nu_i} - \tfrac{1}{\nu} )^2 - \chi_\phi \tfrac{k_B^2 z_i^2}{4 \eta_{\phi i}} ( \tfrac{1}{\nu_i} - \tfrac{1}{\nu} )^2 > 0 \,, \\
  & d_\theta : = \chi_\theta (1 - \eta_\theta - \eta_\theta') b - \chi_m \tfrac{k_B \nu}{4 \eta_m'} \big( \sum_{i=1}^N \tfrac{z_i \delta_i}{\nu_i} \big)^2 - \sum_{i=1}^N \chi_i \tfrac{k_B \delta_i^2}{4 \eta_i \nu_i} - \chi_\phi \tfrac{k_B^2}{4 \eta_\phi} \big( \sum_{i=1}^N \tfrac{z_i \delta_i}{\nu_i} \big)^2 > 0 \,, \\
  & d_m : = \chi_m (1 - \eta_m - \eta_m' ) \tfrac{k_B}{\nu} > 0 \,, \ \tilde{d}_m : = \chi_m \tfrac{1}{\eps} \sum_{i=1}^N \tfrac{z_i^2 \delta_i}{\nu_i} > 0 \,, \ \tilde{d}_\phi : = \chi_\phi \tfrac{k_B \eps}{\nu } > 0 \,,  \\
  & d_\phi : = \chi_\phi \big( \sum_{i=1}^N \tfrac{ z_i^2 \delta_i}{\nu_i} - \sum_{i=1}^N \eta_{\phi i} - \eta_\phi \big) - \chi_\theta \tfrac{ k_B^2}{4 \eta_\theta' b} \big( \sum_{i=1}^N \tfrac{z_i \delta_i}{\nu_i} \big)^2 - \sum_{i=1}^N \chi_i \tfrac{z_i^2 \delta_i^2}{4 \eta_i' k_B \nu_i} > 0 \,.
\end{align*}

\begin{proposition}[A priori estimates]\label{Prop-Apriori}
	Let $s \geq 3$ be an integer. Assume that the function $(n_1, \cdots, n_N, \theta, m, \phi, \u_0)$ is a sufficiently smooth solution on the interval $[0, \tau_0]$ to the perturbed system \eqref{PNPF-Perturbed} with initial data \eqref{IC-perturbed}. Then there is a constant $C_0 > 0$, depending only on $s$, $N$ and the all coefficients, such that
	\begin{equation}\label{Apriori-Est}
	  \begin{aligned}
	    \tfrac{\d}{\d t} \mathscr{E}_s (t) + 2 \mathscr{D}_s (t) \leq C_0 \big( 1 + K({\bf n}) + G({\bf n}) \big) \big( 1 + \mathscr{E}_s^{\frac{s}{2} + 1} (t) \big) \mathscr{E}_s^\frac{1}{2} (t) \mathscr{D}_s (t)
	  \end{aligned}
	\end{equation}
	for all $t \in [0, \tau_0]$, where
	\begin{align}
	  \label{K(n)} & K ({\bf n}) : = \| f ({\bf n}) \|_{L^\infty} + \sum_{i=1}^N \sum_{j=1}^s \| \tfrac{\partial^j f}{\partial n_i^j} ({\bf n}) \|_{L^\infty} \,, \\
	  \label{G(n)} & G({\bf n}) := \sum_{i=1}^N \big\| \tfrac{f({\bf n})}{\delta_i + n_i} \big\|_{L^\infty} + \sum_{i=1}^N \sum_{v=1}^N \sum_{w=1}^s \Big\| \tfrac{\partial^w \big( \tfrac{f({\bf n})}{\delta_i + n_i} \big)}{\partial n_v^w} \Big\|_{L^\infty} \,, \\
	 \label{f(n)} & f ({\bf n}) = \tfrac{1}{a + k_B {\bf c} \cdot {\bf n}} = \tfrac{1}{a + \sum_{i=1}^N k_B c_i n_i} \,,
	\end{align}
	with the vectors ${\bf c} = (c_1, \cdots, c_N)$ and ${\bf n} = (n_1, \cdots, n_N)$ belonging to $\R^N$.
\end{proposition}

Before proving this proposition, we introduce a useful lemma.

\begin{lemma}[Lemma 3.2 of \cite{Jiang-Luo-Tang-2019-M3AS}]\label{Lmm-ChainRule}
	Let $f : \R^N \rightarrow \R$ be a smooth function and ${\bf n} = (n_1, \cdots, n_N) : \R^3 \rightarrow \R^N$ be a vector-valued function belonging to $H^{|\a|}$ for any multi-index $\a \neq 0$. Then,
	\begin{equation}\label{f(n)-equ}
	  \begin{aligned}
	    \partial^\a f ({\bf n}) = \sum_{i=1}^N \sum_{j=1}^{|\a|} \frac{\partial^j f}{\partial n_i^j} ({\bf n}) \sum_{\substack{ \sum_{l=1}^j \a_l = \a \\ |\a_l| \geq 1 }} \prod_{l=1}^j \partial^{\a_l} n_i \,.
	  \end{aligned}
	\end{equation}
	Moreover, if $|\a| \leq s$ $(s \geq 2)$ is further assumed, we deduce from the Sobolev theory that
	\begin{equation}\label{f(n)-norm}
	  \begin{aligned}
	    \| \partial^\a f ({\bf n}) \|_{L^2} \lesssim \sum_{i=1}^N \sum_{j=1}^{s} \| \tfrac{\partial^j f}{\partial n_i^j} ({\bf n}) \|_{L^\infty} \| \nabla n_i \|_{H^{s-1}} \big( 1 + \| \nabla n_i \|^{s-1}_{H^{s-1}} \big) \,.
	  \end{aligned}
	\end{equation}
\end{lemma}

\begin{proof}[Proof of Proposition \ref{Prop-Apriori}]
	For any multi-index $\a \in \mathbb{N}^3$ with $|\a| \leq s$ ($s \geq 3$), we act the derivative operator $\partial^\a$ on the evolutions of $(n_i, \theta, m)$ $(i = 1, \cdots, N)$ in \eqref{PNPF-Perturbed} and employ the similar arguments in deriving the basic energy law \eqref{L2-Summary-L}. We thereby have
	\begin{equation}\label{High-ni-theta-m}
	  \begin{aligned}
	    & \tfrac{1}{2} \tfrac{\d}{\d t} \big( \chi_\phi \eps \| \nabla \partial^\a \phi \|^2_{L^2} + \sum_{i=1}^N \chi_i \| \partial^\a n_i \|^2_{L^2} + a \chi_\theta \| \partial^\a \theta \|^2_{L^2} + \chi_m \| \partial^\a m \|^2_{L^2} \big) \\
	    & + \sum_{i=1}^N d_i \| \nabla \partial^\a n_i \|^2_{L^2} + d_\theta \| \nabla \partial^\a \theta \|^2_{L^2} + d_m \| \nabla \partial^\a m \|^2_{L^2} \\
	    & + \tilde{d}_m \| \partial^\a m \|^2_{L^2} + d_\phi \| \nabla \partial^\a \phi \|^2_{L^2} + \tilde{d}_\phi \| \Delta \partial^\a \phi \|^2_{L^2} \\
	    \leq & \sum_{i=1}^N \chi_i \l \partial^\a R_{n_i} , \partial^\a n_i \r + \chi_\theta \l \partial^\a R_\theta , \partial^\a \theta \r + \chi_m \l \partial^\a R_m , \partial^\a m \r + \chi_\phi \l \partial^\a R_m , \partial^\a \phi \r  \,,
	  \end{aligned}
	\end{equation}
	where the symbols $R_{n_i}$, $R_m$ and $R_\theta$ are defined in \eqref{R-ni}, \eqref{R-m} and \eqref{R-theta}, respectively.
	
	Next we will apply the derivative operator $\partial^\a$ ($|\a| \leq s$) and the Leray projection $\mathcal{P}$ on the third $\u_0$-equation of \eqref{PNPF-Perturbed}. The incompressibility $\nabla \cdot \u_0 = 0$ tells us
	\begin{equation*}
	  \begin{aligned}
	    \Delta \partial^\a \u_0 = \tfrac{1}{\lambda_0} \mathcal{P} \partial^\a R_{\u_0} \,,
	  \end{aligned}
	\end{equation*}
	which implies that by multiplying by $\partial^\a \u_0 $ and integrating by parts over $x \in \R^3$,
	\begin{equation}\label{High-u0}
	  \begin{aligned}
	    \| \nabla \partial^\a \u_0 \|^2_{L^2} = - \tfrac{1}{\lambda_0} \l \partial^\a R_{\u_0}, \partial^\a \u_0 \r \,.
	  \end{aligned}
	\end{equation}
	Here the term $R_{\u_0}$ is defined in \eqref{R-u0}.
	
	We then add the inequalities \eqref{High-ni-theta-m} and \eqref{High-u0} together and sum up for all $|\a| \leq s$. Recalling the definitions of $\mathscr{E}_s (t)$ and $\mathscr{D}_s (t)$ in \eqref{Es} and \eqref{Ds}, respectively, we thereby obtain
	\begin{equation}\label{High-ED}
	  \begin{aligned}
	    \tfrac{1}{2} \tfrac{\d}{\d t} \mathscr{E}_s (t) + & \mathscr{D}_s (t) \leq \sum_{i=1}^N  \chi_i \sum_{|\a| \leq s} \l \partial^\a R_{n_i} , \partial^\a n_i \r + \sum_{|\a| \leq s} \chi_\theta \l \partial^\a R_\theta , \partial^\a \theta \r \\
	    & + \sum_{|\a| \leq s} \chi_m \l \partial^\a R_m , \partial^\a m \r + \sum_{|\a| \leq s} \chi_\phi \l \partial^\a R_m , \partial^\a \phi \r - \sum_{|\a| \leq s} \tfrac{1}{\lambda_0} \l \partial^\a R_{\u_0}, \partial^\a \u_0 \r \,.
	  \end{aligned}
	\end{equation}
	It remains to control the four quantities in the right-hand side of \eqref{High-ED} in terms of the energy $\mathscr{E}_s (t)$ and the dissipative rate $\mathscr{D}_s (t)$. We emphasize that the following embedding inequalities will be frequently used:
	\begin{equation}\label{Embedding-infty}
	  \begin{aligned}
	    & \| f \|_{L^\infty} \leq C_\infty \| f \|_{H^2} \quad \textrm{for some constant } C_\infty > 0 \,, \\
	    & \| f \|_{L^3} \lesssim \| f \|_{L^2}^\frac{1}{2} \| \nabla f \|_{L^2}^\frac{1}{2} \,, \\
	    & \| f \|_{L^4} \lesssim \| f \|^\frac{1}{4}_{L^2} \| \nabla f \|^\frac{3}{4}_{L^2} \lesssim \| f \|_{H^1} \,.
	  \end{aligned}
	\end{equation}
	
	{\em Step 1. Control of the quantity $ \sum_{i=1}^N  \chi_i \sum_{|\a| \leq s} \l \partial^\a R_{n_i} , \partial^\a n_i \r $.}\\
	\vspace*{-3mm}
	
	Recalling the definition of $R_{n_i}$ in \eqref{R-ni}, we have
	\begin{equation}\label{I1I2I3}
	  \begin{aligned}
	    \sum_{|\a| \leq s} \l \partial^\a R_{n_i} , \partial^\a n_i \r = & \underbrace{ - \sum_{|\a| \leq s} \l \partial^\a (\u_0 \cdot \nabla n_i), \partial^\a n_i \r }_{I_1} \ \underbrace{ - \tfrac{ z_i}{\eps \nu_i} \sum_{|\a| \leq s} \l \partial^\a (n_i m) , \partial^\a n_i \r }_{I_2} \\
	    & + \underbrace{ \tfrac{ z_i}{\nu_i} \sum_{|\a| \leq s} \l \partial^\a ( \nabla n_i \cdot \nabla \phi ) , \partial^\a n_i \r }_{I_3} \,.
	  \end{aligned}
	\end{equation}
	Based on the incompressibility $\nabla \cdot \u_0 = 0$, we derive that
	\begin{equation}\label{I1}
	  \begin{aligned}
	    I_1 = & - \sum_{1 \leq |\a| \leq s} \sum_{0 \neq \a' \leq \a} C_\a^{\a'} \l \partial^{\a'} \u_0 \cdot \nabla \partial^{\a - \a'} n_i , \partial^\a n_i \r \\
	    \lesssim & \sum_{1 \leq |\a| \leq s} \sum_{0 \neq \a' \leq \a} \| \partial^{\a'} \u_0 \|_{L^\infty} \| \nabla \partial^{\a - \a'} n_i \|_{L^2} \| \partial^\a n_i \|_{L^2} \\
	    \lesssim & \sum_{1 \leq |\a| \leq s} \sum_{0 \neq \a' \leq \a} \| \partial^{\a'} \u_0 \|_{H^2} \| \nabla \partial^{\a - \a'} n_i \|_{L^2} \| \partial^\a n_i \|_{L^2} \\
	    \lesssim & \| \u_0 \|_{H^{s+2}} \| \nabla n_i \|_{H^s} \| n_i \|_{H^s} \lesssim \mathscr{E}_s^\frac{1}{2} (t) \mathscr{D}_s (t) \,,
	  \end{aligned}
	\end{equation}
	where we have used the H\"older inequality and the first inequality of \eqref{Embedding-infty}. For the term $I_2$, we can infer that
	\begin{equation}\label{I2}
	  \begin{aligned}
	    I_2 = & - \tfrac{ z_i}{\eps \nu_i} \sum_{|\a| \leq s} \sum_{\a' \leq \a} C_\a^{\a'} \l \partial^{\a'} n_i \partial^{\a - \a'} m , \partial^\a n_i \r \\
	    \lesssim &  \sum_{|\a| \leq s} \sum_{\a' \leq \a} \| \partial^{\a'} n_i \|_{L^3} \| \partial^{\a - \a'} m \|_{L^3} \| \partial^\a n_i \|_{L^3} \\
	    \lesssim & \sum_{|\a| \leq s} \sum_{\a' \leq \a} \| \partial^{\a'} n_i \|^\frac{1}{2}_{L^2} \| \nabla \partial^{\a'} n_i \|^\frac{1}{2}_{L^2} \| \partial^{\a - \a'} m \|^\frac{1}{2}_{L^2} \| \nabla \partial^{\a - \a'} m \|^\frac{1}{2}_{L^2} \| \partial^\a n_i \|^\frac{1}{2}_{L^2} \| \nabla \partial^\a n_i \|^\frac{1}{2}_{L^2} \\
	    \lesssim & \| n_i \|_{H^s} \|  \nabla n_i \|_{H^s} \| \nabla m \|^\frac{1}{2}_{H^s} \| m \|^\frac{1}{2}_{H^s} \lesssim \mathscr{E}_s^\frac{1}{2} (t) \mathscr{D}_s (t) \,,
	  \end{aligned}
	\end{equation}
	where the second inequality is derived from the second inequality of \eqref{Embedding-infty}. Next, from the first inequality of \eqref{Embedding-infty}, we deduce that
	\begin{equation}\label{I3}
	  \begin{aligned}
	    I_3 = & \tfrac{ z_i}{\nu_i} \sum_{|\a| \leq s} \sum_{\a' \leq \a} C_\a^{\a'} \l \nabla \partial^{\a'} n_i \cdot \nabla \partial^{\a - \a'} \phi , \partial^\a n_i \r \\
	    \lesssim & \sum_{|\a| \leq s} \Big( \| \nabla n_i \|_{L^\infty} \| \nabla \partial^\a \phi \|_{L^2} + \sum_{0 \neq \a' \leq \a} \| \nabla \partial^{\a'} n_i \|_{L^2} \| \nabla \partial^{\a - \a'} \phi \|_{L^\infty} \Big) \| \partial^\a n_i \|_{L^2} \\
	    \lesssim & \sum_{|\a| \leq s} \Big( \| \nabla n_i \|_{H^1} \| \nabla \partial^\a \phi \|_{L^2} + \sum_{0 \neq \a' \leq \a} \| \nabla \partial^{\a'} n_i \|_{L^2} \| \nabla \partial^{\a - \a'} \phi \|_{H^1} \Big) \| \partial^\a n_i \|_{L^2} \\
	    \lesssim & \| n_i \|_{H^s} \| \nabla n_i \|_{H^s} \big( \| \nabla \phi \|_{H^s} + \| \Delta \phi \|_{H^s} \big) \lesssim \mathscr{E}_s^\frac{1}{2} (t) \mathscr{D}_s (t) \,.
	  \end{aligned}
	\end{equation}
	Consequently, from plugging the bounds \eqref{I1}, \eqref{I2} and \eqref{I3} into the equality \eqref{I1I2I3}, we infer that
	\begin{equation}\label{Control-R-n}
	  \begin{aligned}
	    \sum_{i=1}^N  \chi_i \sum_{|\a| \leq s} \l \partial^\a R_{n_i} , \partial^\a n_i \r \lesssim \mathscr{E}_s^\frac{1}{2} (t) \mathscr{D}_s (t) \,.
	  \end{aligned}
	\end{equation}\\
	
	{\em Step 2. Control of the quantity $\sum_{|\a| \leq s} \chi_m \l \partial^\a R_m , \partial^\a m \r$.} \\
	\vspace*{-3mm}
	
	From the definition of the term $R_m$ in the \eqref{R-m}, we have
	  \begin{align}\label{II1II2II3II4}
	    \no & \sum_{|\a| \leq s} \chi_m \l \partial^\a R_m , \partial^\a m \r = \underbrace{ - \sum_{|\a| \leq s} \chi_m \l \partial^\a (\u_0 \cdot \nabla m) , \partial^\a m \r}_{I\!I_1} \\
	    \no & \underbrace{ - \tfrac{\chi_m }{\eps} \sum_{i=1}^N \tfrac{z_i^2}{\nu_i} \sum_{|\a| \leq s} \l \partial^\a (n_i m) , \partial^\a m \r }_{I\!I_2} + \underbrace{ \chi_m k_B \sum_{i=1}^N \tfrac{z_i}{\nu_i} \sum_{|\a| \leq s} \l \partial^\a \Delta (n_i \theta) , \partial^\a m \r }_{I\!I_3} \\
	    & + \underbrace{ \chi_m \sum_{i=1}^N \tfrac{z_i^2}{\nu_i} \sum_{|\a| \leq s} \l \partial^\a (\nabla n_i \cdot \nabla \phi) , \partial^\a m \r }_{I\!I_4} \,.
	  \end{align}
	By employing the same arguments in deriving the bound \eqref{I1}, we yield that
	\begin{equation}\label{II1}
	  \begin{aligned}
	    I\!I_1 \lesssim \| \u_0 \|_{H^{s+2}} \| \nabla m \|_{H^s} \| m \|_{H^s} \lesssim \mathscr{E}_s^\frac{1}{2} (t) \mathscr{D}_s (t) \,.
	  \end{aligned}
	\end{equation}
	Moreover, it is deduced from the same derivations of the inequality \eqref{I2} that
	\begin{equation}\label{II2}
	  \begin{aligned}
	    I\!I_2 \lesssim \sum_{i=1}^N \| n_i \|^\frac{1}{2}_{H^s} \| \nabla n_i \|^\frac{1}{2}_{H^s} \| m \|_{H^s} \| \nabla m \|_{H^s} \lesssim \mathscr{E}_s^\frac{1}{2} (t) \mathscr{D}_s (t) \,.
	  \end{aligned}
	\end{equation}
	For the term $I\!I_3$, we deduce from the first and the third inequalities of \eqref{Embedding-infty} that
	\begin{align}\label{II3}
	  \no  I\!I_3 = & - \chi_m k_B \sum_{i=1}^N \tfrac{z_i}{\nu_i} \sum_{|\a| \leq s}  \sum_{\a' \leq \a} C_\a^{\a'} \l \nabla \partial^{\a'} n_i \partial^{\a-\a'} \theta + \partial^{\a-\a'} n_i \nabla \partial^{\a'} \theta, \nabla \partial^\a m \r \\
	  \no  \lesssim & \sum_{i=1}^N \sum_{|\a| \leq s} \big( \| \nabla n_i \|_{L^\infty} \| \partial^\a \theta \|_{L^2} + \| \nabla \partial^\a n_i \|_{L^2} \| \theta \|_{L^\infty} \big) \| \nabla \partial^\a m \|_{L^2} \\
	  \no  + & \sum_{i=1}^N \sum_{|\a| \leq s} \big( \| \nabla \theta \|_{L^\infty} \| \partial^\a n_i \|_{L^2} + \| \nabla \partial^\a \theta \|_{L^2} \| n_i \|_{L^\infty} \big) \| \nabla \partial^\a m \|_{L^2} \\
	   \no + & \sum_{i=1}^N \sum_{|\a| \leq s} \sum_{0 \neq \a' < \a} \big( \| \nabla \partial^{\a'} n_i \|_{L^4} \| \partial^{\a - \a'} \theta \|_{L^4} + \| \partial^{\a - \a'} n_i \|_{L^4} \| \nabla \partial^{\a'} \theta \|_{L^4} \big) \| \nabla \partial^\a m \|_{L^2} \\
	   \no \lesssim & \sum_{i=1}^N \sum_{|\a| \leq s} \big( \| \nabla n_i \|_{H^2} \| \partial^\a \theta \|_{L^2} + \| \nabla \partial^\a n_i \|_{L^2} \| \theta \|_{H^2} \big) \| \nabla \partial^\a m \|_{L^2} \\
	   \no + & \sum_{i=1}^N \sum_{|\a| \leq s} \big( \| \nabla \theta \|_{H^2} \| \partial^\a n_i \|_{L^2} + \| \nabla \partial^\a \theta \|_{L^2} \| n_i \|_{H^2} \big) \| \nabla \partial^\a m \|_{L^2} \\
	   \no + & \sum_{i=1}^N \sum_{|\a| \leq s} \sum_{0 \neq \a' < \a} \big( \| \nabla \partial^{\a'} n_i \|_{H^1} \| \partial^{\a - \a'} \theta \|_{H^1} + \| \partial^{\a - \a'} n_i \|_{H^1} \| \nabla \partial^{\a'} \theta \|_{H^1} \big) \| \nabla \partial^\a m \|_{L^2} \\
	   \lesssim & \sum_{i=1}^N \big( \| \nabla n_i \|_{H^s} \| \theta \|_{H^s} + \| \nabla \theta \|_{H^s} \| n_i \|_{H^s} \big) \| \nabla m \|_{H^s} \lesssim \mathscr{E}_s^\frac{1}{2} (t) \mathscr{D}_s (t) \,.
	\end{align}
	Furthermore, the similar derivations of the bound $I_3$ in \eqref{I3} tell us that
	\begin{equation}\label{II4}
	  \begin{aligned}
	    I\!I_4 \lesssim \sum_{i=1}^N \| m \|_{H^s} \| \nabla n_i \|_{H^s} \big( \| \nabla \phi \|_{H^s} + \| \Delta \phi \|_{H^s} \big) \lesssim \mathscr{E}_s^\frac{1}{2} (t) \mathscr{D}_s (t) \,.
	  \end{aligned}
	\end{equation}
	Consequently, from substituting the bounds \eqref{II1}, \eqref{II2}, \eqref{II3} and \eqref{II4} into the relation \eqref{II1II2II3II4}, we deduce that
	\begin{equation}\label{Control-R-m}
	  \begin{aligned}
	    \sum_{|\a| \leq s} \chi_m \l \partial^\a R_m , \partial^\a m \r \lesssim \mathscr{E}_s^\frac{1}{2} (t) \mathscr{D}_s (t) \,.
	  \end{aligned}
	\end{equation}\\
	
	{\em Step 3. Control of the quantity $ - \sum_{|\a| \leq s} \tfrac{1}{\lambda_0} \l \partial^\a R_{\u_0}, \partial^\a \u_0 \r $.} \\
	
	\vspace*{-3mm}
	
	From the definition of the term $R_{\u_0}$ in \eqref{R-u0} and the incompressibility $\nabla \cdot \u_0 = 0$, we straightforwardly compute that
	\begin{equation}\label{III1+2+3}
	  \begin{aligned}
	    - \sum_{|\a| \leq s} \tfrac{1}{\lambda_0} \l \partial^\a R_{\u_0}, \partial^\a \u_0 \r = & \underbrace{ - \sum_{1 \leq |\a| \leq s} \tfrac{1}{\lambda_0} \l \partial^\a (m \nabla \phi), \partial^\a \u_0 \r }_{I\!I\!I_1} \ \underbrace{ - \tfrac{1}{\lambda_0} \l m \nabla \phi, \u_0 \r }_{I\!I\!I_2} \,.
	  \end{aligned}
	\end{equation}
	Via employing the calculus inequalities
	\begin{equation}\label{Calculus-Inq}
	  \begin{aligned}
	    \| f g \|_{H^s} \lesssim \| f \|_{H^s} \| g \|_{H^s}
	  \end{aligned}
	\end{equation}
	for $s > \tfrac{3}{2}$, which can be referred to Lemma 3.4 of \cite{Majda-Bertozzi-2002-BK}, in Page 98 for instance, we know that
	\begin{equation}\label{III1}
	  \begin{aligned}
	    I\!I\!I_1 & \lesssim \sum_{1 \leq |\a| \leq s} \| \partial^\a (m \nabla \phi) \|_{L^2} \| \partial^\a \u_0 \|_{L^2} \lesssim \| m \nabla \phi \|_{H^s} \| \nabla \u_0 \|_{H^s} \\
	    & \lesssim \| m \|_{H^s} \| \nabla \phi \|_{H^s} \| \nabla u_0 \|_{H^s} \lesssim \mathscr{E}_s^\frac{1}{2} (t) \mathscr{D}_s (t) \,.
	  \end{aligned}
	\end{equation}
	Now we control the quantity $I\!I\!I_2$. Since the norm $\| \u_0 \|_{L^2}$ is absent, we will employ the following important relation:
	\begin{equation}\label{Key-Cnc-u0}
	  \begin{aligned}
	    - \tfrac{1}{\lambda_0} \l m \nabla \phi , \u_0 \r = - \tfrac{\eps}{\lambda_0} \l \nabla \phi \otimes \nabla \phi , \nabla \u_0 \r \,,
	  \end{aligned}
	\end{equation}
	derived from $\nabla \cdot \u_0 = 0$, $- \Delta \phi = \tfrac{1}{\eps} m$ and integration by parts over $x \in \R^3$. We therefore see that
	  \begin{align}\label{III2}
	    I\!I\!I_2 = & - \tfrac{\eps}{\lambda_0} \l \nabla \phi \otimes \nabla \phi , \nabla \u_0 \r \lesssim \| \nabla \phi \|^2_{L^3} \| \nabla \u_0 \|_{L^3} \lesssim \| \nabla \phi \|^2_{H^1} \| \nabla \u_0 \|_{H^1} \lesssim \mathscr{E}_s^\frac{1}{2} (t) \mathscr{D}_s (t) \,,
	  \end{align}
	where we have utilized the H\"older inequality and the Sobolev embedding $H^1 (\R^3) \hookrightarrow L^3 (\R^3)$. Therefore, plugging the inequalities \eqref{III1} and \eqref{III2} into the equality \eqref{III1+2+3} reduces to
	\begin{equation}\label{Control-R-u0}
	  \begin{aligned}
	    - \sum_{|\a| \leq s} \tfrac{1}{\lambda_0} \l \partial^\a R_{\u_0}, \partial^\a \u_0 \r \lesssim  \mathscr{E}_s^\frac{1}{2} (t) \mathscr{D}_s (t) \,.
	  \end{aligned}
	\end{equation}
	
	\vspace*{3mm}
	
	{\em Step 4. Control of the quantity $ \sum_{|\a| \leq s} \chi_\theta \l \partial^\a R_\theta , \partial^\a \theta \r $.}
	
	\vspace*{3mm}
	
	Recalling the definition of the term $R_\theta$ in \eqref{R-theta}, we compute that
	  \begin{align}\label{IV-(1-5)}
	    \no \sum_{|\a| \leq s} & \chi_\theta \l \partial^\a R_\theta , \partial^\a \theta \r = \underbrace{ \sum_{i,j=1}^N \tfrac{ k_B^2 c_i z_j \delta_j \chi_\theta}{\nu_j} \sum_{|\a| \leq s} \l \partial^\a ( f ({\bf n}) n_i m ) , \partial^\a \theta \r }_{I\!V_1} \\
	    \no & \underbrace{ - \sum_{i,j=1}^N \tfrac{k_B^3 c_i \chi_\theta}{\nu_j} \sum_{|\a| \leq s} \l \partial^\a ( f ({\bf n}) n_i \Delta n_j ) , \partial^\a \theta \r }_{I\!V_2} \ \underbrace{ - \sum_{i=1}^N b k_B c_i \chi_\theta \sum_{|\a| \leq s} \l \partial^\a ( f ({\bf n}) n_i \Delta \theta ) , \partial^\a \theta \r }_{I\!V_3} \\
	    & \qquad \qquad \underbrace{ - a \chi_\theta \sum_{|\a| \leq s} \l \partial^\a (\u_0 \cdot \nabla \theta) , \partial^\a \theta \r }_{I\!V_4} + \underbrace{ a \chi_\theta \sum_{|\a| \leq s} \l \partial^\a ( f({\bf n}) R_\theta^\star ) , \partial^\a \theta \r }_{I\!V_5} \,,
	  \end{align}
	where $R_\theta^\star$ and $f ({\bf n})$ are defined in \eqref{R-theta-*} and \eqref{f(n)}, respectively.
	
	We first decompose the term $I\!V_1$ into three parts:
	\begin{equation}\label{IV1-1+2+3}
	  \begin{aligned}
	    I\!V_1 = & \underbrace{ \sum_{i,j=1}^N \tfrac{ k_B^2 c_i z_j \delta_j \chi_\theta}{\nu_j} \sum_{|\a| \leq s} \sum_{\a' \leq \a} C_\a^{\a'} \l f({\bf n}) \partial^{\a'} n_i \partial^{\a - \a'} m, \partial^\a \theta \r }_{I\!V_{11}} \\
	    & + \underbrace{ \sum_{i,j=1}^N \tfrac{ k_B^2 c_i z_j \delta_j \chi_\theta}{\nu_j} \sum_{|\a| \leq s} \l \partial^\a f({\bf n}) n_i m, \partial^\a \theta \r }_{I\!V_{12}} \\
	    & + \underbrace{ \sum_{i,j=1}^N \tfrac{ k_B^2 c_i z_j \delta_j \chi_\theta}{\nu_j} \sum_{|\a| \leq s} \sum_{0 \neq \a' < \a} C_\a^{\a'} \l \partial^{\a'} f({\bf n})  \partial^{\a - \a'} ( n_i  m), \partial^\a \theta \r }_{I\!V_{13}} \,.
	  \end{aligned}
	\end{equation}
	It is implied by the H\"older inequality and the second inequality of \eqref{Embedding-infty} that
	  \begin{align}\label{IV11}
	    \no I\!V_{11} \lesssim & \sum_{i=1}^N \sum_{|\a| \leq s} \sum_{\a' \leq \a} \| f ({\bf n}) \|_{L^\infty} \| \partial^{\a'} n_i \|_{L^3} \| \partial^{\a - \a'} m \|_{L^3} \| \partial^\a \theta \|_{L^3} \\
	    \no \lesssim & \sum_{i=1}^N \sum_{|\a| \leq s} \sum_{\a' \leq \a} \| f ({\bf n}) \|_{L^\infty} \| \partial^{\a'} n_i \|^\frac{1}{2}_{L^2} \| \nabla \partial^{\a'} n_i \|^\frac{1}{2}_{L^2} \\
	    \no & \qquad \qquad \quad \times \| \partial^{\a - \a'} m \|^\frac{1}{2}_{L^2} \| \nabla \partial^{\a - \a'} m \|^\frac{1}{2}_{L^2} \| \partial^\a \theta \|^\frac{1}{2}_{L^2} \| \nabla \partial^\a \theta \|^\frac{1}{2}_{L^2} \\
	    \no \lesssim & \sum_{i=1}^N \| f ({\bf n}) \|_{L^\infty} \| n_i \|^\frac{1}{2}_{H^s} \| \theta \|^\frac{1}{2}_{H^s} \| \nabla n_i \|^\frac{1}{2}_{H^s} \| \nabla \theta \|^\frac{1}{2}_{H^s} \| m \|^\frac{1}{2}_{H^s} \| \nabla m \|^\frac{1}{2}_{H^s} \\
	    \lesssim & \| f ({\bf n}) \|_{L^\infty} \mathscr{E}_s^\frac{1}{2} (t) \mathscr{D}_s (t) \,.
	  \end{align}
	From the H\"older inequality, the inequalities in \eqref{Embedding-infty} and the bound \eqref{f(n)-norm} in Lemma \ref{Lmm-ChainRule}, we infer that
	  \begin{align}\label{IV12}
	    \no I\!V_{12} \lesssim & \sum_{i=1}^N \| f ({\bf n}) \|_{L^\infty} \| n_i \|_{L^3} \| m \|_{L^3} \| \theta \|_{L^3} \\
	    \no & + \sum_{i=1}^N \sum_{1 \leq |\a| \leq s} \| \partial^\a f ({\bf n }) \|_{L^2} \| n_i \|_{L^\infty} \| m \|_{L^\infty} \| \partial^\a \theta \|_{L^2} \\
	    \no \lesssim & \sum_{i=1}^N \| f ({\bf n}) \|_{L^\infty} \| n_i \|^\frac{1}{2}_{L^2} \| \nabla n_i \|^\frac{1}{2}_{L^2} \| m \|^\frac{1}{2}_{L^2} \| \nabla m \|^\frac{1}{2}_{L^2} \| \theta \|^\frac{1}{2}_{L^2} \| \nabla \theta \|^\frac{1}{2}_{L^2} \\
	    \no & + \sum_{i=1}^N \sum_{1 \leq |\a| \leq s} \| n_i \|_{H^2} \| m \|_{H^2} \| \partial^\a \theta \|_{L^2} \\
	    \no & \qquad \qquad \times \sum_{v=1}^N \sum_{w=1}^s \| \tfrac{\partial^w f}{\partial n_v^w} ({\bf n}) \|_{L^\infty} \| \nabla n_v \|_{H^{s-1}} \big( 1 + \| \nabla n_v \|^{s-1}_{H^{s-1}} \big) \\
	    \no \lesssim & K({\bf n}) \sum_{i=1}^N \Big( \big( 1 + \| n_i \|^{s-1}_{H^s} \big) \| n_i \|_{H^s} \| \theta \|_{H^s} \| m \|_{H^s} \| \nabla n_i \|_{H^s} \\
	    \no & \qquad \qquad \qquad + \| n_i \|^\frac{1}{2}_{H^s} \| \theta \|^\frac{1}{2}_{H^s} \| \nabla n_i \|^\frac{1}{2}_{H^s} \| \nabla \theta \|^\frac{1}{2}_{H^s} \| m \|^\frac{1}{2}_{H^s} \| \nabla m \|^\frac{1}{2}_{H^s} \Big) \\
	    \lesssim & K ({\bf n}) \big( 1 + \mathscr{E}_s^\frac{s}{2} (t) \big) \mathscr{E}_s^\frac{1}{2} (t) \mathscr{D}_s (t) \,,
	  \end{align}
	where $K ({\bf n})$ is given in \eqref{K(n)}.
	
	We now apply the H\"older inequality, the last inequality in \eqref{Embedding-infty}, the calculus inequality in \eqref{Calculus-Inq} and the inequality \eqref{f(n)-norm} in Lemma \ref{Lmm-ChainRule} to dominate the quantity $I\!V_{13}$. More precisely, we have
	\begin{equation}\label{IV13}
	  \begin{aligned}
	    I\!V_{13} \lesssim & \sum_{i=1}^N \sum_{|\a| \leq s} \sum_{0 \neq \a' < \a} \| \partial^{\a'} f ({\bf n}) \|_{L^4} \| \partial^{\a - \a'} (n_i m) \|_{L^4} \| \partial^\a \theta \|_{L^2} \\
	    \lesssim & \sum_{i=1}^N \sum_{|\a| \leq s} \sum_{0 \neq \a' < \a} \| \partial^{\a'} f ({\bf n}) \|_{H^1} \| \partial^{\a - \a'} (n_i m) \|_{H^1} \| \partial^\a \theta \|_{L^2} \\
	    \lesssim & \sum_{i=1} \| \nabla f ({\bf n}) \|_{H^{s-1}} \| n_i m \|_{H^s} \| \theta \|_{H^s} \\
	    \lesssim & K ({\bf n}) \sum_{i=1} \big( 1 + \| n_i \|^{s-1}_{H^s} \big) \| n_i \|_{H^s} \| \theta \|_{H^s} \| \nabla n_i \|_{H^s} \| m \|_{H^s} \\
	    \lesssim & K ({\bf n}) \big( 1 + \mathscr{E}_s^\frac{s-1}{2} (t) \big) \mathscr{E}_s (t) \mathscr{D}_s (t) \,.
	  \end{aligned}
	\end{equation}
	Collecting the all relations \eqref{IV1-1+2+3}, \eqref{IV11}, \eqref{IV12} and \eqref{IV13}, we immediately obtain
	\begin{equation}\label{IV1}
	  \begin{aligned}
	    I\!V_1 \lesssim K ({\bf n}) \big( 1 + \mathscr{E}_s^\frac{s}{2} (t) \big) \mathscr{E}_s^\frac{1}{2} (t) \mathscr{D}_s (t) \,,
	  \end{aligned}
	\end{equation}
	where the $K({\bf n})$ is defined in \eqref{K(n)}.
	
	Secondly, we devote ourselves to control the term $I\!V_2$. We split it into three parts:
	\begin{align}\label{IV2-1+2+3}
	  \no I\!V_2 = & \underbrace{ \sum_{i,j=1}^N \tfrac{k_B^3 c_i \chi_\theta}{\nu_j} \sum_{|\a| \leq s} \l \nabla [ f({\bf n}) n_i \partial^\a \theta ], \nabla \partial^\a n_j \r }_{I\!V_{21}} \\
	  \no & \underbrace{ - \sum_{i,j=1}^N \tfrac{k_B^3 c_i \chi_\theta}{\nu_j} \sum_{1 \leq |\a| \leq s} \l \partial^\a ( f({\bf n}) n_i ) \Delta n_j , \partial^\a \theta \r }_{I\!V_{22}} \\
	  & \underbrace{ - \sum_{i,j=1}^N \tfrac{k_B^3 c_i \chi_\theta}{\nu_j} \sum_{|\a| \leq s} \sum_{0 \neq \a' < \a} C_\a^{\a'} \l \partial^{\a'} ( f({\bf n}) n_i ) \Delta \partial^{\a - \a'} n_j , \partial^\a \theta \r }_{I\!V_{23}} \,.
	\end{align}
	It is derived from the inequality \eqref{f(n)-norm} in Lemma \ref{Lmm-ChainRule} and the first inequality in \eqref{Embedding-infty} that
	  \begin{align}\label{IV21}
	    \no I\!V_{21} \lesssim & \sum_{i,j=1}^N \sum_{|\a| \leq s} \| \nabla \partial^\a n_j \|_{L^2} \big( \| \nabla \partial^\a \theta \|_{L^2} \| f ({\bf n}) \|_{L^\infty} \| n_i \|_{L^\infty} \\
	    \no & \qquad \qquad \qquad + \| \nabla f ({\bf n}) \|_{L^\infty} \| n_i \|_{L^\infty} \| \partial^\a \theta \|_{L^2} + \| f ({\bf n}) \|_{L^\infty} \| \nabla n_i \|_{L^\infty} \| \partial^\a \theta \|_{L^2} \big) \\
	    \no \lesssim & K({\bf n}) \sum_{i,j=1}^N \| \nabla n_j \|_{H^s} \Big( \| n_i \|_{H^s} \| \nabla \theta \|_{H^s} + \| \theta \|_{H^s} \| \nabla n_i \|_{H^s} \\
	    \no & \qquad \qquad \qquad \qquad \qquad \qquad + ( 1 + \| \nabla n_i \|_{H^2}^2 ) \| \theta \|_{H^s} \| n_i \|_{H^s} \| \nabla n_i \|_{H^2} \Big) \\
	    \no \lesssim & K({\bf n}) \sum_{i=1}^N ( 1 + \| n_i \|^3_{H^s} ) ( \| n_i \|_{H^s} + \| \theta \|_{H^s} ) \big( \| \nabla n_i \|^2_{H^s} + \| \nabla \theta \|^2_{H^s} \big) \\
	    \lesssim &  K({\bf n}) \big( 1 + \mathscr{E}_s^\frac{3}{2} (t) \big) \mathscr{E}_s^\frac{1}{2} (t) \mathscr{D}_s (t)\,,
	  \end{align}
	where the integer $s \geq 3$ is required and $K({\bf n})$ is given in \eqref{K(n)}. For the quantity $I\!V_{22}$, we deduce from the H\"older inequality, the first inequality in \eqref{Embedding-infty}, the calculus inequality \eqref{Calculus-Inq} and the inequality \eqref{f(n)-norm} in Lemma \ref{Lmm-ChainRule} that
	  \begin{align}\label{IV22}
	    \no I\!V_{22} \lesssim & \sum_{i,j=1}^N \sum_{1 \leq |\a| \leq s} \| \partial^\a ( f({\bf n}) n_i ) \|_{L^2} \| \Delta n_j \|_{L^\infty} \| \partial^\a \theta \|_{L^2} \\
	    \no \lesssim & \sum_{i,j=1}^N \sum_{1 \leq |\a| \leq s} \| \partial^\a ( f({\bf n}) n_j ) \|_{L^2} \| \Delta n_j \|_{H^2} \| \partial^\a \theta \|_{L^2} \\
	    \no \lesssim & \sum_{i,j=1}^N \| \nabla ( f({\bf n}) n_i ) \|_{H^{s-1}} \| \nabla n_j \|_{H^3} \| \theta \|_{H^s} \\
	    \no \lesssim & \sum_{i,j=1}^N \big( ( \| \nabla
	    f({\bf n} ) n_i ) \|_{H^{s-1}} + \| f({\bf n}) \nabla n_i \|_{H^{s-1}} \big) \| \nabla n_j \|_{H^s} \| \theta \|_{H^s} \\
	    \no \lesssim & K({\bf n}) \sum_{i=1}^N ( 1 + \| n_i \|^{s-1}_{H^s} ) \| n_i \|_{H^s} \| \theta \|_{H^s} \| \nabla n_i \|^2_{H^s} \\
	    \lesssim & K({\bf n}) \big( 1 + \mathscr{E}_s^\frac{s-1}{2} (t) \big) \mathscr{E}_s (t) \mathscr{D}_s (t) \,.
	  \end{align}
	Based on the H\"older inequality, the last inequality in \eqref{Embedding-infty} and the inequality \eqref{f(n)-norm} in Lemma \ref{Lmm-ChainRule}, the quantity $I\!V_{23}$ can be bounded by
	\begin{align}\label{IV23}
	  \no I\!V_{23} \lesssim & \sum_{i,j=1} \sum_{|\a| \leq s} \sum_{0 \neq \a' < \a} \| \partial^{\a'} ( f({\bf n}) n_i ) \|_{L^4} \| \Delta \partial^{\a - \a'} n_j \|_{L^2} \| \partial^\a \theta \|_{L^4} \\
	  \no \lesssim & \sum_{i,j=1} \sum_{|\a| \leq s} \sum_{0 \neq \a' < \a} \| \partial^{\a'} ( f({\bf n}) n_i ) \|_{H^1} \| \nabla n_j \|_{H^s} \| \partial^\a \theta \|_{H^1} \\
	  \no \lesssim & \sum_{i,j=1} \| f ({\bf n}) \|_{H^s} \| n_i \|_{H^s} \| \nabla n_j \|_{H^s} \| \nabla \theta \|_{H^s} \\
	  \no \lesssim & K ({\bf n}) \sum_{i=1}^N ( 1 + \| n_i \|^{s-1}_{H^s} ) \| n_i \|^2_{H^s} \| \nabla n_i \|_{H^s} \| \nabla \theta \|_{H^s} \\
	  \lesssim & K({\bf n}) \big( 1 + \mathscr{E}_s^\frac{s-1}{2} (t) \big) \mathscr{E}_s (t) \mathscr{D}_s (t) \,.
	\end{align}
	Therefore, plugging the bounds \eqref{IV21}, \eqref{IV22} and \eqref{IV23} into the equality \eqref{IV2-1+2+3} reduces to
	\begin{equation}\label{IV2}
	  \begin{aligned}
	    I\!V_{2} \lesssim K({\bf n}) \big( 1 + \mathscr{E}_s^\frac{s}{2} (t) \big) \mathscr{E}_s^\frac{1}{2} (t) \mathscr{D}_s (t) \,.
	  \end{aligned}
	\end{equation}
	Moreover, from the similar arguments in estimating the bound \eqref{IV2}, we can deduce that
	\begin{equation}\label{IV3}
	  \begin{aligned}
	    I\!V_{3} \lesssim K({\bf n}) \big( 1 + \mathscr{E}_s^\frac{s}{2} (t) \big) \mathscr{E}_s^\frac{1}{2} (t) \mathscr{D}_s (t) \,.
	  \end{aligned}
	\end{equation}
	Furthermore, by employing the same arguments in \eqref{I1}, we have
	\begin{equation}\label{IV4}
	  \begin{aligned}
	    I\!V_4 \lesssim \| \u_0 \|_{H^{s+2}} \| \nabla \theta \|_{H^s} \| \theta \|_{H^s} \lesssim \mathscr{E}_s^\frac{1}{2} (t) \mathscr{D}_s (t) \,.
	  \end{aligned}
	\end{equation}
	
	Finally, we dominate the quantity $I\!V_5 = a \sum_{|\a| \leq s} \l \partial^\a ( f({\bf n}) R_\theta^\star ) , \partial^\a \theta \r$, where the term $R_\theta^\star$ is given in \eqref{R-theta-*} and $f ({\bf n})$ is mentioned as in \eqref{f(n)}. The term $I\!V_5$ can be specifically expressed as
	  \begin{align}\label{IV5-1...9}
	    \no I\!V_5 = & \underbrace{ \sum_{i=1} \tfrac{k_B^2}{\nu_i} \sum_{|\a| \leq s} \l \partial^\a [ f ({\bf n}) \Delta (n_i \theta) ] , \partial^\a \theta \r }_{I\!V_{51}} \ \underbrace{ - \sum_{i=1}^N \tfrac{ k_B z_i}{\eps \nu_i} \sum_{|\a| \leq s} \l \partial^\a [ f({\bf n}) n_i m ] , \partial^\a \theta \r }_{I\!V_{52}} \\
	    \no & + \underbrace{ \sum_{i=1}^N \tfrac{ k_B c_i z_i}{\nu_i} \sum_{|\a| \leq s} \l \partial^\a [ f({\bf n}) (\delta_i + n_i) \nabla \phi \cdot \nabla \theta ] , \partial^\a \theta \r }_{I\!V_{53}} \\
	    \no & + \underbrace{ \lambda_0 \sum_{|\a| \leq s} \l \partial^\a ( f({\bf n}) |\nabla \u_0|^2 ) , \partial^\a \theta \r }_{I\!V_{54}} \ \underbrace{ - \sum_{i=1}^N \tfrac{ k_B z_i}{\eps \nu_i} \sum_{|\a| \leq s} \l \partial^\a [ f({\bf n}) (\delta_i + n_i) m \theta ] , \partial^\a \theta \r }_{I\!V_{55}} \\
	    \no & + \underbrace{ \sum_{i=1}^N \tfrac{k_B^2 c_i}{\nu_i} \sum_{|\a| \leq s} \l \partial^\a [ f({\bf n}) \nabla (n_i + \delta_i \theta + n_i \theta ) \cdot \nabla \theta ] , \partial^\a \theta \r }_{I\!V_{56}} \\
	    \no & + \underbrace{ \sum_{i=1}^N \tfrac{k_B^2}{\nu_i} \sum_{|\a| \leq s} \l \partial^\a [ f({\bf n}) \Delta ( n_i + \delta_i \theta + n_i \theta ) \theta ] , \partial^\a \theta \r }_{I\!V_{57}} \\
	    \no & \underbrace{ - \sum_{i=1}^N \tfrac{k_B^2}{\nu_i} \sum_{|\a| \leq s} \l \partial^\a \big[ \tfrac{f({\bf n})}{\delta_i + n_i} (1 + \theta) \nabla n_i \cdot \nabla ( n_i + \delta_i \theta + n_i \theta ) \big] , \partial^\a \theta \r }_{I\!V_{58}} \\
	    & + \underbrace{ \sum_{i=1}^N \tfrac{1}{\nu_i} \sum_{|\a| \leq s} \l \partial^\a \big[ \tfrac{f({\bf n})}{\delta_i + n_i} | k_B \nabla ( n_i + \delta_i \theta + n_i \theta ) + e z_i ( \delta_i + n_i ) \nabla \phi |^2 \big] , \partial^\a \theta \r }_{I\!V_{59}} \,.
	  \end{align}
	  From the analogous arguments in estimating the quantities \eqref{IV1} and \eqref{IV2}, one can easily deduce the following bounds:
	  \begin{align}
	    \no I\!V_{51} \lesssim & K({\bf n}) \sum_{i=1}^N ( 1 + \| n_i \|^{s-1}_{H^s} ) \| \theta \|_{H^s} \| n_i \|_{H^s} ( \| \nabla n_i \|^2_{H^s} + \| \nabla \theta \|^2_{H^s} ) \\
	    \label{IV51} \lesssim & K({\bf n}) \big( 1 + \mathscr{E}_s^\frac{s-1}{2} (t) \big) \mathscr{E}_s (t) \mathscr{D}_s (t) \,, \\
	    \no I\!V_{52} \lesssim & K({\bf n}) \sum_{i=1}^N ( 1 + \| n_i \|^{s-1}_{H^s} ) \| n_i \|_{H^s} \| \theta \|_{H^s} \| \nabla n_i \|_{H^s} \| m \|_{H^s} \\
	    \label{IV52} \lesssim & K({\bf n}) \big( 1 + \mathscr{E}_s^\frac{s-1}{2} (t) \big) \mathscr{E}_s (t) \mathscr{D}_s (t) \,, \\
	    \no I\!V_{53} \lesssim & K({\bf n}) \sum_{i=1}^N ( 1 + \| n_i \|^s_{H^s} ) \| n_i \|_{H^s} \| \theta \|_{H^s} \| \nabla \phi \|_{H^s} \| \nabla \theta \|_{H^s} \\
	    \label{IV53} \lesssim & K({\bf n}) \big( 1 + \mathscr{E}_s^\frac{s}{2} (t) \big) \mathscr{E}_s (t) \mathscr{D}_s (t) \,, \\
	    \no I\!V_{54} \lesssim & K({\bf n}) \sum_{i=1}^N ( 1 + \| n_i \|^{s-1}_{H^s} ) \| n_i \|_{H^s} \| \theta \|_{H^s} \| \nabla \u_0 \|^2_{H^s} \\
	    \label{IV54} \lesssim & K({\bf n}) \big( 1 + \mathscr{E}_s^\frac{s-1}{2} (t) \big) \mathscr{E}_s (t) \mathscr{D}_s (t) \,, \\
	    \no I\!V_{55} \lesssim & K({\bf n}) \sum_{i=1}^N ( 1 + \| n_i \|^s_{H^s} ) \| \theta \|_{H^s} ( \| \nabla n_i \|_{H^s} + \| \nabla \theta \|_{H^s} ) \| m \|_{H^s} \\
	    \label{IV55} \lesssim & K({\bf n}) \big( 1 + \mathscr{E}_s^\frac{s}{2} (t) \big) \mathscr{E}_s^\frac{1}{2} (t) \mathscr{D}_s (t) \,, \\
	    \no I\!V_{56} \lesssim & K({\bf n}) \sum_{i=1}^N ( 1 + \| n_i \|^{s-1}_{H^s} ) ( 1 + \| n_i \|_{H^s} + \| \theta \|_{H^s} ) \| \theta \|_{H^s} ( \| \nabla n_i \|^2_{H^s} + \| \nabla \theta \|^2_{H^s} ) \\
	    \label{IV56} \lesssim & K({\bf n}) \big( 1 + \mathscr{E}_s^\frac{s}{2} (t) \big) \mathscr{E}_s^\frac{1}{2} (t) \mathscr{D}_s (t) \,, \\
	    \no I\!V_{57} \lesssim & K({\bf n}) \sum_{i=1}^N ( 1 + \| n_i \|^{s-1}_{H^s} ) ( 1 + \| n_i \|_{H^s} + \| \theta \|_{H^s} ) \| \theta \|^2_{H^s} ( \| \nabla n_i \|^2_{H^s} + \| \nabla \theta \|^2_{H^s} ) \\
	    \label{IV57} \lesssim & K({\bf n}) \big( 1 + \mathscr{E}_s^\frac{s}{2} (t) \big) \mathscr{E}_s (t) \mathscr{D}_s (t) \,.
	  \end{align}
	  Here, for simplicity, we omit the details of the derivations. Moreover, if we replace the function $f ({\bf n})$ by $ \tfrac{f({\bf n})}{\delta_i + n_i} $ in the arguments of the estimating $I\!V_1$ and $I\!V_2$ in \eqref{IV1} and \eqref{IV2}, respectively, we can also analogously estimate the term $I\!V_{58}$ as follows:
	  \begin{equation}\label{IV58}
	    \begin{aligned}
	      I\!V_{58} \lesssim & G({\bf n}) \sum_{i=1}^N ( 1 + \| n_i \|^{s-1}_{H^s} ) ( 1 + \| n_i \|^2_{H^s} + \| \theta \|^2_{H^s} ) \| n_i \|_{H^s} \| \theta \|_{H^s} ( \| \nabla n_i \|^2_{H^s} + \| \nabla \theta \|^2_{H^s} ) \\
	      \lesssim & G({\bf n}) \big( 1 + \mathscr{E}_s^\frac{s+1}{2} (t) \big) \mathscr{E}_s (t) \mathscr{D}_s (t) \,,
	    \end{aligned}
	  \end{equation}
	  where $G({\bf n})$ are defined in \eqref{G(n)}. By applying the similar arguments in estimating the quantity $I\!V_{58}$ in \eqref{IV58}, we can control the term $I\!V_{59}$ as follows:
	  \begin{equation}\label{IV59}
	    \begin{aligned}
	      I\!V_{59} \lesssim & G({\bf n}) \sum_{i=1}^N ( 1 + \| n_i \|^{s-1}_{H^s} ) ( 1 + \| n_i \|_{H^s} + \| \theta \|_{H^s} ) \\
	      & \qquad \qquad \times \| \theta \|_{H^s} \big( \| \nabla n_i \|^2_{H^s} + \| \nabla \theta \|^2_{H^s} + \| \nabla \phi \|^2_{H^s} \big) \\
	      \lesssim & G({\bf n}) \big( 1 + \mathscr{E}_s^\frac{s}{2} (t) \big) \mathscr{E}_s^\frac{1}{2} (t) \mathscr{D}_s (t) \,.
	    \end{aligned}
	  \end{equation}
	  Then, we substitute the inequalities \eqref{IV51}, \eqref{IV52}, \eqref{IV53}, \eqref{IV54}, \eqref{IV55}, \eqref{IV56}, \eqref{IV57}, \eqref{IV58} and \eqref{IV59} into the equality \eqref{IV5-1...9}, so that we obtain
	  \begin{equation}\label{IV5}
	    \begin{aligned}
	      IV_5 \lesssim \big( K({\bf n}) + G({\bf n}) \big) \big( 1 + \mathscr{E}_s^{\tfrac{s}{2} + 1} (t) \big) \mathscr{E}_s^\frac{1}{2} (t) \mathscr{D}_s (t) \,.
	    \end{aligned}
	  \end{equation}
	  If the bounds \eqref{IV1}, \eqref{IV2}, \eqref{IV3}, \eqref{IV4} and \eqref{IV5} are further plugged into the relation \eqref{IV-(1-5)}, we get
	  \begin{equation}\label{Control-R-theta}
	    \begin{aligned}
	      \sum_{|\a| \leq s} \chi_\theta \l \partial^\a R_\theta , \partial^\a \theta \r \lesssim \big( 1 + K({\bf n}) + G({\bf n}) \big) \big( 1 + \mathscr{E}_s^{\frac{s}{2} + 1} (t) \big) \mathscr{E}_s^\frac{1}{2} (t) \mathscr{D}_s (t) \,,
	    \end{aligned}
	  \end{equation}
	  where the symbols $K({\bf n})$ and $G({\bf n})$ are defined in \eqref{K(n)} and \eqref{G(n)}, respectively.
	
	  \vspace*{3mm}
	
	  {\em Step 5. Control the quantity $\sum_{|\a| \leq s} \chi_\phi \l \partial^\a R_m , \partial^\a \phi \r$.}
	
	  \vspace*{3mm}
	
	  Recalling that the energy $\mathscr{E}_s (t)$ and dissipative rate $\mathscr{D}_s (t)$ do not involve the norm $\| \phi \|_{L^2}$, we therefore shall avoid the $L^2$-norm of $\phi$ (without derivative of $\phi$) when dominating the quantity $\sum_{|\a| \leq s} \chi_\phi \l \partial^\a R_m , \partial^\a \phi \r$. Then
	  \begin{equation}\label{V1V2}
	    \begin{aligned}
	      \sum_{|\a| \leq s} \chi_\phi \l \partial^\a R_m , \partial^\a \phi \r = \underbrace{ \chi_\phi \l R_m , \phi \r }_{V_1} + \underbrace{ \sum_{1 \leq |\a| \leq s} \chi_\phi \l \partial^\a R_m , \partial^\a \phi \r }_{V_2} \,,
	    \end{aligned}
	  \end{equation}
	  where the quantity $V_1$, who explicitly involves a $\phi$ without derivative, should be dealt carefully. Recalling the $R_m$ is defined in \eqref{R-m} and $- \Delta \phi = \tfrac{1}{\eps} m$, we can derive
	  \begin{equation*}
	    \begin{aligned}
	      \l R_m , \phi \r = & \langle - \u_0 \cdot \nabla m + k_B \sum_{i=1}^N \tfrac{z_i}{\nu_i} \Delta (n_i \theta) - \tfrac{1}{\eps} \sum_{i=1}^N \tfrac{z_i^2}{\nu_i} n_i m + \sum_{i=1}^N \tfrac{z_i^2}{\nu_i} \nabla n_i \cdot \nabla \phi , \phi \rangle \\
	      = & \eps \l \u_0 \cdot \nabla \Delta \phi , \phi \r + \sum_{i=1}^N \tfrac{z_i^2}{\nu_i} \l n_i \Delta \phi + \nabla n_i \cdot \nabla \phi , \phi \r - k_B \sum_{i=1}^N \tfrac{z_i}{\nu_i} \l \nabla (n_i \theta) , \nabla \phi \r \,.
	    \end{aligned}
	  \end{equation*}
	  The integration by parts over $x \in \R^3$ and the divergence-free of $\u_0$ reduce to
	  \begin{equation}\label{Key-Rlts}
	    \begin{aligned}
	      & \l \u_0 \cdot \nabla \Delta \phi, \phi \r = - \l \u_0 , \nabla \tfrac{1}{2} |\nabla \phi|^2 \r - \l \nabla \otimes \nabla \phi , \nabla \u_0 \r = - \l \nabla \phi \otimes \nabla \phi , \nabla \u_0 \r \,, \\
	      & \l n_i \Delta \phi + \nabla n_i \cdot \nabla \phi , \phi \r = - \l n_i , |\nabla \phi|^2 \r \,.
	    \end{aligned}
	  \end{equation}
	  Then we have
	  \begin{equation}\label{V1}
	    \begin{aligned}
	      V_1 = & \chi_\phi \l R_m , \phi \r \lesssim \| \nabla \phi \|^2_{L^3} \big( \| \nabla \u_0 \|_{L^3} + \sum_{i=1}^N \| n_i \|_{L^3} \big) + \sum_{i=1}^N \| \nabla (n_i \theta) \|_{L^2} \| \nabla \phi \|_{L^2} \\
	      \lesssim & \| \nabla \phi \|^2_{H^1} ( \| \nabla \u_0 \|_{H^1} + \sum_{i=1}^N \| n_i \|_{H^1} ) \\
	       & + \sum_{i=1}^N \| \nabla \phi \|_{H^1} ( \| \nabla \theta \|_{H^1} + \| \nabla n_i \|_{H^1} ) ( \| n_i \|_{H^2} + \| \theta \|_{H^2} ) \\
	      \lesssim & \mathscr{E}_s^\frac{1}{2} (t) \mathscr{D}_s (t) \,.
	    \end{aligned}
	  \end{equation}
	  For the quantity $V_2$, we have
	  \begin{equation}\label{V21V22}
	    \begin{aligned}
	      V_2 = & \underbrace{ \eps \sum_{1 \leq |\a| \leq s} \chi_\phi \l \u_0 \cdot \nabla \Delta \partial^\a \phi , \partial^\a \phi \r }_{V_{21}} \\
	      & \left.
	      \begin{aligned}
	        & - \sum_{1 \leq |\a| \leq s} \chi_\phi \sum_{0 \neq \a' < \a} \l \partial^{\a '} \u_0 \cdot \nabla \partial^{\a - \a '} m , \partial^\a \phi \r \\
	        & - k_B \sum_{1 \leq |\a| \leq s} \sum_{i=1}^N \tfrac{z_i}{\nu_i} \l \nabla \partial^\a (n_i \theta) , \nabla \partial^\a \phi \r - \tfrac{1}{\eps} \sum_{1 \leq |\a| \leq s} \sum_{i=1}^N \tfrac{z_i^2}{\nu_i} \l \partial^\a (n_i m) , \partial^\a \phi \r \\
	        & + \sum_{1 \leq |\a| \leq s} \sum_{i=1}^N \tfrac{z_i^2}{\nu_i} \l \partial^\a (\nabla n_i \cdot \nabla \phi ), \partial^\a \phi \r
	      \end{aligned}
	      \right\} {V_{22}} \,,
	    \end{aligned}
	  \end{equation}
	  where the equation $- \Delta \phi = \tfrac{1}{\eps} m$ has been used. The key relation \eqref{Key-Rlts} reduce to
	  \begin{equation}\label{V21}
	    \begin{aligned}
	      V_{21} = & - \chi_\phi \eps \sum_{1 \leq |\a| \leq s} \l \nabla \partial^\a \phi \otimes \nabla \partial^\a \phi , \nabla \u_0 \r \lesssim \| \nabla \partial^\a \phi \|^2_{L^2} \| \nabla \u_0 \|_{L^\infty} \\
	      \lesssim & \| \nabla \phi \|^2_{H^s} \| \nabla \u_0 \|_{H^s} \lesssim \mathscr{E}_s^\frac{1}{2} (t) \mathscr{D}_s (t) \,.
	    \end{aligned}
	  \end{equation}
	  By the Sobolev embedding theory, one easily derives that
	    \begin{align}\label{V22}
	      \no V_{22} \lesssim & \| \nabla \phi \|^2_{H^s} \| \nabla \u_0 \|_{H^s} + \sum_{i=1}^N \| \nabla n_i \|_{H^s} \| \nabla \theta \|_{H^s} \| \nabla \phi \|_{H^s} \\
	      \no & + \sum_{i=1}^N \| n_i \|_{H^s} \| m \|_{H^s} \| \nabla \phi \|_{H^s} + \sum_{i=1}^N \| \nabla n_i \|_{H^s} \| \nabla \phi \|^2_{H^s} \\
	      \lesssim & \mathscr{E}_s^\frac{1}{2} (t) \mathscr{D}_s (t) \,.
	    \end{align}
	  Plugging \eqref{V21} and \eqref{V22} into \eqref{V21V22} implies that
	  \begin{equation}\label{V2}
	    \begin{aligned}
	      V_2 \lesssim \mathscr{E}_s^\frac{1}{2} (t) \mathscr{D}_s (t) \,.
	    \end{aligned}
	  \end{equation}
	  It is derived from substituting \eqref{V1} and \eqref{V2} into \eqref{V1V2} that
	  \begin{equation}\label{Control-Rm-phi}
	    \begin{aligned}
	      \sum_{|\a| \leq s} \chi_\phi \l \partial^\a R_m , \partial^\a \phi \r \lesssim \mathscr{E}_s^\frac{1}{2} (t) \mathscr{D}_s (t) \,.
	    \end{aligned}
	  \end{equation}
	  Consequently, the inequalities \eqref{High-ED}, \eqref{Control-R-n}, \eqref{Control-R-m}, \eqref{Control-R-u0}, \eqref{Control-R-theta} and \eqref{Control-Rm-phi} imply the a priori estimate inequality \eqref{Apriori-Est}, and the proof of Proposition \ref{Prop-Apriori} is finished.	
\end{proof}

\begin{proof}[\bf Proof of Theorem \ref{Main-Thm}]
	Based on the a priori estimate \eqref{Apriori-Est} in Proposition \ref{Prop-Apriori}, we now prove the main result of current paper by employing the continuity arguments.
	
	We first deal with the quantities $K({\bf n})$ and $G({\bf n})$ defined in \eqref{K(n)} and \eqref{G(n)}, respectively. One easily observes that there is a constant $\beta > 0$ such that
	\begin{equation}\label{KGn-1}
	  \begin{aligned}
	    K({\bf n}) + G({\bf n}) \leq \beta \| f({\bf n}) \|_{L^\infty} + \beta \sum_{i=1}^N \| \tfrac{f({\bf n})}{\delta_i + n_i} \|_{L^\infty} + \beta \| f({\bf n}) \|^{s+1}_{L^\infty} + \beta \sum_{i=1}^N \| \tfrac{f({\bf n})}{\delta_i + n_i} \|^{s+1}_{L^\infty} \,,
	  \end{aligned}
	\end{equation}
	where $ f({\bf n}) $ is given in \eqref{f(n)}. From the first Sobolev inequality in \eqref{Embedding-infty}, we deduce that
	\begin{equation}\label{f(n)-Bnd-1}
	  \begin{aligned}
	    a + k_B \sum_{i=1}^N c_i n_i \geq & a - k_B \sum_{i=1}^N c_i \| n_i \|_{L^\infty} \geq a - k_B C_\infty \max_{1 \leq i \leq N} \{c_i\} \sum_{i=1}^N \| n_i \|_{H^s} \\
	    \geq & a - k_B C_\infty \max_{1 \leq i \leq N} \{c_i\} \Big( \sum_{i=1}^N \tfrac{1}{\chi_i} \Big)^\frac{1}{2} \Big( \sum_{i=1}^N \chi_i \| n_i \|^2_{H^s} \Big)^\frac{1}{2} \\
	    \geq & a - \gamma_1 \mathscr{E}_s^\frac{1}{2} (t) \,,
	  \end{aligned}
	\end{equation}
	where $a > 0$ is given in \eqref{Coeffs-ab}, $C_\infty > 0$ is mentioned as in \eqref{Embedding-infty} and
	\begin{equation*}
	  \begin{aligned}
	    \gamma_1 = k_B C_\infty \max_{1 \leq i \leq N} \{c_i\} \Big( \sum_{i=1}^N \tfrac{1}{\chi_i} \Big)^\frac{1}{2} > 0 \,.
	  \end{aligned}
	\end{equation*}
	Similarly in \eqref{f(n)-Bnd-1}, one immediately has
	\begin{equation}\label{f(n)-Bnd-2}
	  \begin{aligned}
	    \delta_i + n_i \geq \delta_i - \tfrac{C_\infty}{\sqrt{\chi_i}} \mathscr{E}_s^\frac{1}{2} (t) \geq \delta_0 - \gamma_2 \mathscr{E}_s^\frac{1}{2} (t) \,,
	  \end{aligned}
	\end{equation}
	where
	\begin{equation*}
	  \begin{aligned}
	    \delta_0 = \min_{1 \leq i \leq N} \{ \delta_i \} > 0 \,, \ \gamma_2 = \max_{1 \leq i \leq N} \big\{ \tfrac{C_\infty}{\sqrt{\chi_i}} \big\} > 0 \,.
	  \end{aligned}
	\end{equation*}
	
	It is easy to know that
	\begin{equation*}
	  \begin{aligned}
	    \mathscr{E}_s (0) \leq \max \{ \tfrac{ \chi_\phi}{\eps}, \chi_1, \cdots, \chi_N, \chi_m z_1^2, \cdots, \chi_m z_N^2, a \chi_\theta \} E^\inn : = \gamma_0 E^\inn \,.
	  \end{aligned}
	\end{equation*}
	We now take $\xi_1 = \tfrac{1}{16} \min \big\{ \tfrac{a^2}{\gamma_0 \gamma_1^2}, \tfrac{\delta_0^2}{\gamma_0 \gamma_2^2 , } \big\} > 0$ such that if $E^\inn \leq \xi_1$, we derive from the relations \eqref{f(n)-Bnd-1} and \eqref{f(n)-Bnd-2} that
	\begin{equation*}
	  \begin{aligned}
	    & a + k_B \sum_{i=1}^N c_i n_i^\inn \geq a - \gamma_1 \mathscr{E}_s^\frac{1}{2} (0) \geq a - \gamma_1 \sqrt{\gamma_0 E^\inn} \geq \tfrac{3}{4} a > \tfrac{1}{2} a > 0 \,, \\
	    & \delta_i + n_i^\inn \geq \delta_0 - \gamma_2 \mathscr{E}_s^\frac{1}{2} (0) \geq \tfrac{3}{4} \delta_0 > \tfrac{1}{2} \delta_0 > 0 \,.
	  \end{aligned}
	\end{equation*}
	Thus, the relation \eqref{KGn-1} tells us that
	\begin{equation*}
	  \begin{aligned}
	    K({\bf n^\inn}) + G({\bf n^\inn}) \leq & \underbrace{ \tfrac{\beta}{\tfrac{3}{4} a} \big( 1 + \tfrac{1}{\tfrac{3}{4} \delta_0} \big) + \tfrac{\beta}{(\tfrac{3}{4} a)^{s+1}} \big( 1 + \big( \tfrac{1}{\tfrac{3}{4} \delta_0} \big)^{s+1} \big) }_{\gamma_3} \\
	    < & \underbrace{ \tfrac{\beta}{\tfrac{1}{2} a} \big( 1 + \tfrac{1}{\tfrac{1}{2} \delta_0} \big) + \tfrac{\beta}{(\tfrac{1}{2} a)^{s+1}} \big( 1 + \big( \tfrac{1}{\tfrac{1}{2} \delta_0} \big)^{s+1} \big) }_{\gamma_4} \,.
	  \end{aligned}
	\end{equation*}
	We further take
	\begin{equation*}
	  \begin{aligned}
	    \xi_0 = \min \Big\{ \xi_1, \tfrac{1}{ 4 \gamma_0 C_0^2 ( 1 + \gamma_4 )^2 ( 1 + (4 \xi_1)^{\frac{s}{2} + 1} )^2 }  \Big\} > 0 \,,
	  \end{aligned}
	\end{equation*}
	such that if $E^\inn \leq \xi_0$, we have
	\begin{equation}\label{InEngBnd}
	  \begin{aligned}
	    \mathscr{E}_s (0) \leq \gamma_0 \xi_1\,, \ C_0 ( 1 + \gamma_4 ) \big( 1 + \mathscr{E}_s^{\frac{s}{2} + 1} (0) \big) \mathscr{E}_s^\frac{1}{2} (0) \leq \tfrac{1}{2} \,.
	  \end{aligned}
	\end{equation}
	Now we introduce a number
	\begin{equation*}
	  \begin{aligned}
	    T^\star : = \sup \Bigg\{ \tau > 0; \sup_{t \in [0, \tau]} \mathscr{E}_s (t) \leq & 4 \gamma_0 \xi_1 \textrm{ and } C_0 (1 + \gamma_4) \sup_{t \in [0, \tau]} \big[ \big( 1 + \mathscr{E}_s^{\frac{s}{2} + 1} (t) \big) \mathscr{E}_s^\frac{1}{2} (t) \big] \leq 1 \Bigg\} \,.
	  \end{aligned}
	\end{equation*}
	Then the initial energy bound \eqref{InEngBnd} and the continuity of the energy functional $\mathscr{E}_s (t)$ imply that $T^\star > 0$.
	
	We claim that $T^\star = + \infty$. Indeed, if $T^\star < + \infty$, we easily deduce that for all $t \in [0, T^\star]$,
	\begin{equation*}
	  \begin{aligned}
	    1 + K({\bf n}) + G({\bf n}) \leq 1 + \gamma_4 \,,
	  \end{aligned}
	\end{equation*}
	which implies that the a priori estimate \eqref{Apriori-Est} reduces to
	\begin{equation*}
	  \begin{aligned}
	    \tfrac{\d}{\d t} \mathscr{E}_s (t) + 2 \mathscr{D}_s (t) \leq C_0 ( 1 + \gamma_4 ) \big( 1 + \mathscr{E}_s^{\frac{s}{2} + 1} (t) \big) \mathscr{E}_s^\frac{1}{2} (t) \mathscr{D}_s (t) \leq \mathscr{D}_s (t) \,,
	  \end{aligned}
	\end{equation*}
	namely,
	\begin{equation*}
	  \begin{aligned}
	     \tfrac{\d}{\d t} \mathscr{E}_s (t) + \mathscr{D}_s (t) \leq 0
	  \end{aligned}
	\end{equation*}
	for all $t \in [0, T^\star]$. Then, integrating the above inequality over $[0,t] \subseteq [0, T^\star]$ tells us that for all $t \in [0, T^\star]$,
	\begin{equation}\label{Global-1}
	  \begin{aligned}
	    \mathscr{E}_s (t) + \int_0^t \mathscr{D}_s (\tau) \d \tau \leq \mathscr{E}_s (0) \leq \gamma_0 \xi_0 < 4 \gamma_0 \xi_1 \,,
	  \end{aligned}
	\end{equation}
	which yields that
	\begin{equation*}
	  \begin{aligned}
	    C_0 (1 + \gamma_4) \big[ \big( 1 + \mathscr{E}_s^{\frac{s}{2} + 1} (t) \big) \mathscr{E}_s^\frac{1}{2} (t) \big] \leq \tfrac{1}{2} < 1 \,.
	  \end{aligned}
	\end{equation*}
	From the continuity of the energy functional $\mathscr{E}_s (t)$ and the definition of the number $T^\star$, we then deduce that there is a small $  \varsigma > 0$ such that for all $t \in [ 0, T^\star + \varsigma ]$,
	\begin{equation*}
	  \begin{aligned}
	    \mathscr{E}_s (t) \leq 4 \gamma_0 \xi_1 \,, \ C_0 (1 + \gamma_4) \big[ \big( 1 + \mathscr{E}_s^{\frac{s}{2} + 1} (t) \big) \mathscr{E}_s^\frac{1}{2} (t) \big] \leq 1 \,,
	  \end{aligned}
	\end{equation*}
	which contracts to the definition of $T^\star$. Thus $T^\star = + \infty$. So, the energy bound \eqref{Global-1} implies the global energy bound \eqref{Global-Energy}.
	
	Since $m \in L^\infty (\R^+; H^s)$, $\nabla m \in L^2 (\R^+; H^s)$ and $- \Delta \phi = \tfrac{1}{\eps} m$, one infers from the global energy bound \eqref{Global-Energy} that
	\begin{equation}\label{Rg-phi}
	  \nabla \phi, \Delta \phi \in L^\infty (\R^+; H^s) \,, \ \nabla \Delta \phi \in L^2 (\R^+; H^{s+2}) \,.
	\end{equation}
	It is derived from \eqref{Rg-phi}, the global energy bound \eqref{Global-Energy} and the $\u_0$-equation in \eqref{PNPF-Perturbed} that
	\begin{equation*}
	  \begin{aligned}
	    \lambda_0 \Delta \u_0 = \mathcal{P} ( m \nabla \phi ) \in L^\infty (\R^+; H^s) \cap L^2_{loc} (\R^+; H^{s+1}) \,,
	  \end{aligned}
	\end{equation*}
	where $\mathcal{P}$ is the Leray projection. Then, combining the global energy bound \eqref{Global-Energy}, we obtain
	$$\nabla \u_0 \in L^2 (\R^+; H^s) \,, \Delta \u_0 \in L^\infty (\R^+; H^s) \cap L^2_{loc} (\R^+; H^{s+1}) \,.$$
	Moreover, since $R_{\u_0} = \sum_{i=1}^N k_B \nabla (n_i \theta) + m \nabla \phi \in L^\infty (\R^+; H^{s-1}) \cap L^2_{loc} (\R^+; H^s)$, the $\u_0$-equation in \eqref{PNPF-Perturbed} reads
	\begin{equation*}
	  \nabla P_0 = \lambda_0 \Delta \u_0 - \sum_{i=1}^N k_B \nabla ( n_i + \delta_i \theta ) - R_{\u_0} \in L^\infty (\R^+; H^{s-1}) \cap L^2_{loc} (\R^+; H^s) \,.
	\end{equation*}
	Finally, based on the global energy bound \eqref{Global-Energy}, the $\u_i$-equation in \eqref{PNPF-1} implies that
	\begin{equation*}
	  \begin{aligned}
	    \u_i = \u_0 - \tfrac{k_B}{\nu_i} \nabla \theta + \tfrac{k_B}{\nu_i (\delta_i + n_i)} \nabla n_i (1 + \theta) - \tfrac{ z_i}{\nu_i} \nabla \phi\ (1 \leq i \leq N)
	  \end{aligned}
	\end{equation*}
	satisfy $\nabla \u_i \in L^2 (\R^+; H^{s-1})$. Consequently, the proof of Theorem \ref{Main-Thm} is finished.	
\end{proof}

\appendix

\section{Detailed derivations of the reformulations \eqref{PNPF-2} and \eqref{PNPF-Perturbed}}\label{Sec:Appendix}

In this section, we will give the details on deriving the reformulation \eqref{PNPF-2} of the original PNPF system \eqref{PNPF-1} and the perturbed equations \eqref{PNPF-Perturbed}. More precisely, we introduce the following lemma.

\begin{lemma}\label{Lmm-Syst-Transform}
	Let $m = \sum_{j=1}^N z_j \rho_j$ be the total electric charge. Then the system \eqref{PNPF-1} can be rewritten as the form \eqref{PNPF-2}. Furthermore, if we consider the following perturbations \eqref{Perturbations}, then the functions $( n_1, \cdots, n_N, \theta, m, \phi,  \u_0, P_0)$ subjects to the equations \eqref{PNPF-Perturbed}.
\end{lemma}

\begin{proof}
	We first derive the form \eqref{PNPF-2} from  the original system \eqref{PNPF-1}. It is easy to derive from the third equation of \eqref{PNPF-1} and the definition of $m = \sum_{j=1}^N z_j \rho_j$ that
	\begin{equation}
	\begin{aligned}
	- \Delta \phi = \tfrac{1}{\eps} m \,.
	\end{aligned}
	\end{equation}
	From plugging the second equation into the last second equation of \eqref{PNPF-1}, we deduce that
	\begin{equation}
	\begin{aligned}
	\lambda_0 \Delta \u_0 = \nabla P_0 + \sum_{i=1}^N k_B \nabla (\rho_i T) + m \nabla \phi \,.
	\end{aligned}
	\end{equation}
	Moreover, we derive from substituting the second equation into the first equation of \eqref{PNPF-1} that
	\begin{equation}
	\begin{aligned}
	\partial_t \rho_i + \u_0 \cdot \nabla \rho_i - \tfrac{k_B}{\nu_i} \Delta (\rho_i T) - \nabla \cdot \big( \tfrac{z_i}{\nu_i} \rho_i \nabla \phi \big) = 0
	\end{aligned}
	\end{equation}
	for $ i = 1, \cdots, N $.
	
	We next derive the evolution of the total electric charge $m$. From the definition of $m$ and the first equation of \eqref{PNPF-1}, we deduce that
	\begin{equation}
	\begin{aligned}
	\partial_t m = \sum_{j=1}^N z_j \partial_t \rho_j = - \sum_{j=1}^N \nabla \cdot (z_j \rho_j \u_j) \,.
	\end{aligned}
	\end{equation}
	From the second equation of \eqref{PNPF-1} and $m = \sum_{j=1}^N z_j \rho_j$, we deduce that
	\begin{equation}\label{rhoiui}
	\begin{aligned}
	\rho_i \u_i = \rho_i \u_0 - \tfrac{k_B}{\nu_i} \nabla (\rho_i T) - \tfrac{ z_i}{\nu_i} \rho_i \nabla \phi \,,
	\end{aligned}
	\end{equation}
	and then
	\begin{equation}
	\begin{aligned}
	\sum_{i=1}^N z_i \rho_i \u_i = m \u_0 - \tfrac{k_B}{\nu} \Delta (mT) - k_B \sum_{i=1}^N (\tfrac{1}{\nu_i} - \tfrac{1}{\nu}) z_i \nabla (\rho_i T) - \sum_{i=1}^N \tfrac{z_i^2}{\nu_i} \rho_i \nabla \phi \,,
	\end{aligned}
	\end{equation}
	where $\nu = \tfrac{N}{\sum_{j=1}^N \tfrac{1}{\nu_j}} > 0$ is the harmonic average of the viscosities $\nu_1$, $\nu_2$, $\cdots$, $\nu_N$. We thereby obtain
	\begin{equation}
	\begin{aligned}
	\partial_t m + \u_0 \cdot \nabla m - \tfrac{k_B}{\nu} \Delta (m T) = \sum_{i=1}^N \tfrac{z_i^2}{\nu_i} \nabla \cdot ( \rho_i \nabla \phi ) + k_B \sum_{i=1}^N (\tfrac{1}{\nu_i} - \tfrac{1}{\nu}) z_i \Delta (\rho_i T) \,,
	\end{aligned}
	\end{equation}
	and
	\begin{equation}
	\begin{aligned}
	\partial_t \rho_i + \u_0 \cdot \nabla \rho_0 - \tfrac{k_B}{\nu_i} \Delta (\rho_i T) - \nabla \cdot ( \tfrac{z_i}{\nu_i} \rho_i \nabla \phi ) = 0
	\end{aligned}
	\end{equation}
	for $i = 1, \cdots, N$, where $\nabla \cdot \u_0 = 0$ is utilized.
	
	It remains to compute the evolution of the temperature $T$.
	
	The second equation of \eqref{PNPF-1} firstly tells
	\begin{equation}
	\begin{aligned}
	\nu_i \rho_i | \u_i - \u_0 |^2 = \frac{1}{\nu_i \rho_i} | k_B \nabla (\rho_i T) + z_i \rho_i \nabla \phi |^2 \,.
	\end{aligned}
	\end{equation}
	Moreover, the relation \eqref{rhoiui} implies that
	\begin{equation}
	\begin{aligned}
	& \Big( \sum_{i=0}^N k_B c_i \rho_i \u_i \Big) \cdot \nabla T \\
	= & \sum_{i=0}^N k_B c_i \rho_i \u_0 \cdot \nabla T - \sum_{i=1}^N \tfrac{k_B^2 c_i}{\nu_i} \nabla (\rho_i T) \cdot \nabla T - \sum_{i=1}^N \tfrac{ k_B c_i z_i}{\nu_i} \rho_i \nabla \phi \cdot \nabla T \,,
	\end{aligned}
	\end{equation}
	and
	\begin{equation}
	\begin{aligned}
	\Big( \sum_{i=1}^N k_B \rho_i \nabla \cdot \u_i \Big) = \Big( - \sum_{i=1}^N \tfrac{k_B^2}{\nu_i} \Delta (\rho_i T) + \sum_{i=1}^N \frac{k_B^2}{\nu_i \rho_i} \nabla \rho_i \cdot \nabla (\rho_i T) + \sum_{i=1}^N \tfrac{k_B z_i}{\nu_i} \rho_i m  \Big) T \,.
	\end{aligned}
	\end{equation}
	Consequently, the forth equation of \eqref{PNPF-1} reads
	\begin{equation}
	\begin{aligned}
	\Big( \sum_{i=0}^N k_B c_i \rho_i \Big) \partial_t T - k \Delta T + \sum_{i=0}^N k_B c_i \rho_i \u_0 \cdot \nabla T - \sum_{i=1}^N \tfrac{k_B^2 c_i}{\nu_i} \nabla (\rho_i T) \cdot \nabla T \\
	= \lambda_0 |\nabla \u_0|^2 + \sum_{i=1}^N \tfrac{1}{\nu_i \rho_i} \left| k_B \nabla (\rho_i T) + z_i \rho_i \nabla \phi \right|^2 + \sum_{i=1}^N \tfrac{k_B C_i z_i}{\nu_i} \rho_i \nabla \phi \cdot \nabla T \\
	\qquad + \sum_{i=1}^N \Big( \tfrac{k_B^2}{\nu_i} \Delta (\rho_i T) - \tfrac{k_B^2}{\nu_i \rho_i} \nabla_i \cdot \nabla (\rho_i T) - \tfrac{k_B z_i}{\eps \nu_i} \rho_i m \Big) T \,.
	\end{aligned}
	\end{equation}
	Then we obtain the formulate \eqref{PNPF-2} of the PNPF system.
	
	Finally, from plugging the perturbations \eqref{Perturbations}, i.e., $\rho_i = \delta_i + n_i$, $T = 1 + \theta$, into the equations \eqref{PNPF-2}, we easily deduce the perturbed system \eqref{PNPF-Perturbed}, and then the proof of Lemma \ref{Lmm-Syst-Transform} is completed.
\end{proof}


\section*{Acknowledgment}

The author N. J. appreciates Prof. Chun Liu introduced this problem to him and provide many insights during the preparation of this work.

\bigskip

\bibliography{reference}

\end{document}